\documentclass{amsart}
\usepackage{amssymb,amsmath,amsthm,verbatim,amstext,eucal,amscd}

\usepackage{mathtools}
\usepackage{verbatim}
\usepackage{enumitem}
\usepackage[all]{xy}

\numberwithin{equation}{section}

\let\aH=\H

\newtheorem{lemma}{Lemma}[section] 
\newtheorem{proposition}[lemma]{Proposition}
\newtheorem{theorem}[lemma]{Theorem}

\newtheorem{corollary}[lemma]{Corollary}

\newcommand{\FLEX}{\relax}
\newcommand{\flex}[1]{\renewcommand{\FLEX}{#1}}
\newtheorem{flexthm}[lemma]{\FLEX}

\newtheorem{prop}[lemma]{Proposition}

\theoremstyle{definition}
\newtheorem{definition}[lemma]{Definition}
\newtheorem{example}[lemma]{Example}

\theoremstyle{remark}
\newtheorem{remark}[lemma]{Remark}

\newcommand{\bh}{\ensuremath{{\mathcal B}({\mathcal H})}}

\newcommand{\cstaralg}{$C^*$-algebra}

\providecommand{\dual}[1]{\ensuremath{#1^{\#}}}
\providecommand{\ddual}[1]{\ensuremath{#1^{\#\#}}}

\newcommand{\dom}{\operatorname{dom}}

\newcommand{\dstext}[1]{\quad\text{#1}\quad}

\newcommand{\innerprod}[1]{\left\langle #1\right\rangle}

\newcommand{\meet}{\mathop{
     \mathchoice{\vcenter{\hbox{\huge
           $\bigwedge$}}}{\wedge}{\wedge}{\wedge}}
     \displaylimits}

\newcommand{\norm}[1]{\left\|{#1}\right\|}

\providecommand{\qed}%
{\hfill \vrule height5pt width4pt depth1pt \vspace{+2.00ex}}

\newcommand{\ran}{\operatorname{range}}

\newcommand{\spn}{\operatorname{span}}

\newcommand{\bet}{\mbox{$\beta$}}

\newcommand{\bbA}{{\mathbb{A}}}

\newcommand{\bbC}{{\mathbb{C}}}

\newcommand{\bbN}{{\mathbb{N}}}

\newcommand{\bbT}{{\mathbb{T}}}

  \newcommand{\A}{{\mathcal{A}}}
  
  \newcommand{\B}{{\mathcal{B}}}
  \newcommand{\C}{{\mathcal{C}}}
  
  \newcommand{\D}{{\mathcal{D}}}
  
  \newcommand{\E}{{\mathcal{E}}}
  \newcommand{\F}{{\mathcal{F}}}
  \newcommand{\G}{{\mathcal{G}}}
 \let\aH=\H  
\renewcommand{\H}{{\mathcal{H}}}
  
  \newcommand{\J}{{\mathcal{J}}}
\renewcommand{\L}{{\mathcal{L}}}
  \newcommand{\M}{{\mathcal{M}}}
  \newcommand{\N}{{\mathcal{N}}}
\renewcommand{\O}{{\mathcal{O}}}
\renewcommand{\P}{{\mathcal{P}}}

\renewcommand{\S}{{\mathcal{S}}}
  \newcommand{\T}{{\mathcal{T}}}

  \newcommand{\Z}{{\mathcal{Z}}}

\newcommand{\fA}{{\mathfrak{A}}}
\newcommand{\fB}{{\mathfrak{B}}}

\newcommand{\fH}{{\mathfrak{H}}}

\newcommand{\fJ}{{\mathfrak{J}}}

\newcommand{\fP}{{\mathfrak{P}}}

\newcommand{\msd}{\operatorname{\textsc{msd}}}
\newcommand{\mtr}{\operatorname{\textsc{mtr}}}
\newcommand{\js}{\veebar}

 \newcommand{\pil}{\alpha_\ell{}}
 \newcommand{\fn}{\mathfrak{n}}

\newcommand{\<}{\langle}
\renewcommand{\>}{\rangle}

\newcommand{\sge}{\sb q}
\newcommand{\Gr}{\operatorname{Graph}}

\newcommand{\cocycle}{\sigma}

\begin{document}

\title[Von Neumann Algebras and Inverse Semigroups]
{Von Neumann
  Algebras and  Extensions of Inverse Semigroups} 
\author[A. P. Donsig]{Allan P. Donsig}\address{Dept. of
Mathematics\\ University of Nebraska-Lincoln\\ Lincoln, NE\\
68588-0130 } \email{adonsig1@math.unl.edu}
\author[A. H. Fuller]{Adam H. Fuller} \address{Dept. of
Mathematics\\ University of Nebraska-Lincoln\\ Lincoln, NE\\
68588-0130 } \email{afuller7@math.unl.edu} 
\author[D.R. Pitts]{David
R. Pitts} \address{Dept. of Mathematics\\ University of
Nebraska-Lincoln\\ Lincoln, NE\\ 68588-0130}
\email{dpitts2@math.unl.edu}
\date{\today}


\begin{abstract}
  In the 1970s, Feldman and Moore classified separably acting von
  Neumann algebras containing Cartan MASAs using measured equivalence
  relations and $2$-cocycles on such equivalence relations. In this
  paper, we give a new classification in terms of extensions of
  inverse semigroups.  Our approach is more algebraic in character and
  less point-based than that of Feldman-Moore.  As an application, we
  give a restatement of the spectral theorem for bimodules in terms of
  subsets of inverse semigroups. We also show how our viewpoint leads
  naturally to a description of maximal subdiagonal algebras.
\end{abstract}

\maketitle

\section{Introduction}\label{S: Intro}
Every abelian von Neumann algebra is isomorphic to $L^\infty(X,\mu)$
for a suitable measure space $(X,\mu)$.  Because of this, the theory
of von Neumann algebras is often described as ``non-commutative
integration.''  In a pair of landmark papers, Feldman and
Moore~\cite{FeldmanMooreErEqReI,FeldmanMooreErEqReII} pursued this
analogy further.  They showed that if $\D\simeq L^\infty(X,\mu)$ is a
Cartan MASA in a separably acting von Neumann algebra $\M$, then there
is a Borel equivalence relation $R\subseteq X\times X$ and a
$2$-cocycle $c$ on $R$ such that $\M$ is isomorphic to an algebra $M(R,c)$
consisting of certain measurable functions on $R$ and $\D$ is
isomorphic to the algebra $D(R,c)$ of functions supported on the
diagonal $\{(x,x): x\in R\}$ of $R$.  The multiplication in $M(R,c)$
is essentially matrix multiplication twisted by the cocycle $c$.
Feldman and Moore further show that the isomorphism classes of pairs
$(\M,\D)$ with $\D$ a Cartan MASA in a separably acting von Neumann
algebra $\M$ is in bijective correspondence with the family of
equivalence classes of pairs $(R,c)$ where $c$ is a 2-cocycle on the
measured equivalence relation $R$.
Twisting the multiplication by a cocycle originated in the work of
Zeller-Meier for crossed products of von Neumann
algebras~\cite[Section~8]{Zeller-MeierPrCrCstAlGrAu}, which was itself
an extension of the group-measure construction.  The Cartan pairs of
Feldman and Moore include these crossed products.

Feldman and Moore's work may be characterized as ``point-based'' in
the sense that the basic objects used in their construction are
functions determined up to null sets on appropriate measure spaces.
As a result of the measure theory involved, the Feldman-Moore work is
restricted to equivalence relations with countable equivalence classes
and to von Neumann algebras with separable predual.  Furthermore,
their work demands considerable measure-theoretic prowess.

The goal of the present paper is to recast the Feldman-Moore work in
algebraic terms.  We bypass the measured equivalence relations used by
Feldman and Moore and instead start with an axiomatization of the
inverse semigroups which arise from measured equivalence relations.
Here is a brief description of these inverse semigroups.  Starting
with a measured equivalence relation, Feldman and Moore consider the
family $\S$ of all partial Borel isomorphisms $\phi: X\rightarrow X$
whose graph, $\Gr(\phi):=\{(\phi(x), x): x\in X\}$, is a subset of
$R$.  With composition product, $\S$ becomes an inverse semigroup and
the characteristic function of the set $\Gr(\phi)$ becomes a partial
isometry in $M(R,c)$.  The strong-$*$ closure $\G$ of the inverse
semigroup generated by such isometries and $\bbT I$ is an inverse
semigroup of partial isometries which generates $M(R,c)$.  Further,
$\G$ is an inverse semigroup extension of $\S$.
We axiomatize the class of the inverse semigroups arising as partial
Borel isomorphisms whose graph lies in a measured equivalence
relation; we call members of this class of inverse semigroups
\textit{Cartan inverse monoids}.

Lausch~\cite{LauschCoInSe} has developed a theory of extensions of
inverse semigroups which parallels the theory of extensions of groups.
In particular, Lausch shows that there is a natural notion of
equivalence of extensions, and that up to equivalence, the family of
extensions of a given inverse semigroup by an abelian inverse
semigroup is parametrized by a $2$-cohomology group.  We replace the
2-cocycle on $R$ appearing in the Feldman-Moore work  with an
extension of the Cartan inverse monoid $\S$ by the abelian inverse
semigroup of partial isometries in the \cstaralg\ generated by the
idempotents of $\S$.  From this data, we construct a Cartan MASA in a
von Neumann algebra of the extension.  This is accomplished in
Theorem~\ref{T: indeppsi}.

We show in Theorem~\ref{T: same data} that any Cartan MASA $\D$ in a
von Neumann algebra $\M$ determines an extension of the type mentioned
in the previous paragraph.  In combination, Theorems~\ref{T: same
data} and~\ref{T: indeppsi} show that these constructions are inverses
of each other up to equivalence.  Thus, we obtain the desired
algebraic version of the Feldman-Moore work.

We note that our constructions apply to any pair $(\M,\D)$ consisting of a Cartan MASA $\D$ in the von Neumann algebra $\M$.
We require neither $\M$ to act separably, nor any hypothesis on Cartan inverse monoids which would correspond to countable equivalence classes of measured equivalence relations.

In constructing a Cartan pair from an extension, we build a
representation of the Cartan inverse monoid analogous to the
Stinespring representation of $\pi\circ E$, where $E :
\M\rightarrow\D$ is the conditional expectation and $\pi$ is a
representation of $\D$ on the Hilbert space $\H$.  Since the inverse
semigroup has no innate linear structure (as $\M$ does) we use an
operator-valued reproducing kernel Hilbert space approach.  The
construction of the corresponding reproducing kernel uses the order
structure of $\S$ arising from the action of the idempotents of $\S$.
This action should be viewed as the semigroup analogue of the bimodule
action of $\D$ on $\M$.

An important application of the Feldman-Moore construction is to
characterize the $\D$-bimodules of $\M$ in terms of suitable subsets
of $R$.  For Bures-closed $\D$-bimodules, such a characterization was
obtained by Cameron, Pitts, and
Zarikian~\cite[Theorem~2.5.1]{CameronPittsZarikianBiCaMASAvNAlNoAlMeTh}.
In Theorem~\ref{T: Spectral Theorem} below, we reformulate this
characterization in terms of subsets of $\S$, which we call spectral
sets.  As a result, we describe maximal subdiagonal algebras of $\M$
which contain $\D$ in terms of spectral sets.  In particular, this
provides a proof of~\cite[Theorem~3.5]{MuhlySaitoSolelCoTrOpAl} that
avoids the weak-$*$-closed Spectral Theorem for Bimodules
\cite[Theorem~2.5]{MuhlySaitoSolelCoTrOpAl} whose proof unfortunately
is incomplete.

\section{Preliminaries} \label{S: prelim}
We begin with a discussion of the necessary ideas about Boolean
algebras and inverse semigroups.

\subsection{Stone's representation theorem}\label{Ss: Stone}
Let $\L$ be a Boolean algebra, and let $\widehat{\L}$ be the character
space of $\L$, that is, the set of
all lattice homomorphisms of $\L$ into the two element lattice
$\{0,1\}$.   For each $e\in \E$, let
\begin{equation*} G_e = \{ \rho\in\widehat{L} \colon \rho(e)=1\}.
\end{equation*} Stone's representation theorem shows the sets $\{G_e
\colon e\in \L\}$ form a basis for a compact Hausdorff topology on
$\widehat{\L}$ (\cite{StoneApThBoRiGeTo}, or see e.g.,
\cite{GivantHalmosInBoAl}).  In this topology, each set $G_e$ is
clopen.  Thus Stone's theorem represents $\L$  as the algebra
of clopen sets in $\widehat{\L}$.  Equivalently, $\L$ can be viewed as
the lattice of projections in $C(\widehat{\L})$.

We now show that $C(\widehat{\L})$ is the universal \cstaralg\ of
$\L$.
\begin{definition} Let $\L$ be a Boolean algebra.  A
  \textit{representation} of $\L$ is a map $\pi:\L\rightarrow
  \text{proj}(\B)$ of $\L$ into the projection lattice of a
 \cstaralg\ $\B$ such that for every
  $s,t\in\L$, $\pi(s\meet t)=\pi(s)\pi(t)$.
\end{definition}

\begin{proposition}\label{P: universal1} Let $\L$ be a Boolean algebra
with character space $\widehat{\L}$.  For each $s\in\L$, let
$\widehat{s}\in C(\widehat{\L})$ be the Gelfand transform,
$\widehat{s}(\rho)=\rho(s)$.  Then $C(\widehat{\L})$ has the following
universal property: if $\B$ is a \cstaralg\ and $\theta:\L\rightarrow
\B$ is a representation such that $\theta(\L)$ generates $\B$ as a
\cstaralg, then there exists a unique $*$-epimorphism
$\alpha:C(\widehat{\L})\rightarrow \B$ such that for every $s\in\L$,
\[\theta(s)=\alpha(\widehat{s}).\]
\end{proposition}

\begin{proof} By the definition of representation, $\theta(\L)$ is a
  commuting family of projections, and, since $\theta(\L)$ generates
  $\B$, $\B$ is abelian. For $\rho\in\widehat{\B}$,
  $\rho\circ\theta\in\widehat{\L}$.  Moreover, the dual map
  $\dual{\theta}:\hat{\B}\rightarrow\hat{\L}$ given by
  $\widehat{\B}\ni \rho\mapsto \rho\circ \theta$ is continuous.  Hence
  there is a $*$-homomorphism $\alpha:C(\hat{\L})\rightarrow \B$ given
  by  
\[\widehat{\alpha(f)}=f\circ\dual{\theta}.\]   For $s\in\L$ and
$\rho\in\widehat{\B}$, we have
\[\widehat{\alpha(\widehat{s})}(\rho)
=\widehat{s}(\rho\circ\theta)=\rho(\theta(s))=\widehat{\theta(s)}
(\rho),\] so that $\theta(s)=\alpha(\widehat{s})$.  Since $\theta(\L)$
generates $\B$, the image of $\alpha$ is dense in $\B$, whence
$\alpha$ is onto.

Suppose $\alpha_1:C(\widehat{\L})\rightarrow \B$ is another
$*$-epimorphism of $C(\widehat{\L})$ onto $\B$ such that
$\alpha_1(\widehat{s})=\theta(s)$ for every $s\in\L$.  Letting $\A$ be
the $*$-algebra generated by $\{\widehat{s}:s\in\L\}$, we find $\A$
separates points of $\widehat{\L}$ and contains the constant
functions.  The Stone-Weierstrass Theorem shows $\A$ is dense in
$C(\widehat{\L})$.  Since $\alpha_1|_\A=\alpha|_\A$, we conclude that
$\alpha_1=\alpha$.
\end{proof}

\subsection{Inverse semigroups}
We discuss some results and
definitions in the theory of inverse semigroups.  For a comprehensive
text on inverse semigroups, see Lawson \cite{LawsonInSe}.

A semigroup $\S$ is an \emph{inverse semigroup} if there is a unique
inverse operation on $\S$.  That is, for every $s\in\S$ there is a
unique element $s^\dag$ in $\S$ satisfying
\begin{equation*} ss^\dag s=s \text{ and } s^\dag ss^\dag =s^\dag.
\end{equation*} Two elements $s,t\in\S$ are \textit{orthogonal} if
$s^\dag t=t s^\dag=0$.  An inverse semigroup $\S$ is an
\textit{inverse monoid} if $\S$ has a multiplicative unit; we usually
denote the unit with the symbol $1$.

We denote the idempotents in $\S$ by $\E(\S)$.  The idempotents of an
inverse semigroup form an abelian inverse subsemigroup.  Further,
$\E(\S)$ determines the \textit{natural partial order} on $\S$: given
$s,t\in \S$, write $s\leq t$ if there is an idempotent $e\in\S$ such
that
\begin{equation*} s=te.
\end{equation*} 
We will often use the notation $(\S,\leq)$ when we ``forget'' the
multiplication on $\S$ and simply consider $\S$  as a set with
this natural partial order.

For $s,t\in\S$, we will use $s\wedge t$ for the greatest lower bound
of $\{s,t\}$, if it exists.  Likewise, $s\vee t$ will denote the least
upper bound.  In general inverse semigroups, $s\vee t$ and $s\meet t$
need not exist.  If for any $s,t\in \S$, $s\wedge t$ exists in $\S$,
$(\S,\leq)$ is a \textit{meet semilattice}.

Idempotents of the form $s^\dag t \wedge 1$ are called \emph{fixed
  point idempotents} by Leech \cite{LeechInMoWiANaSeOr}.  When
$(\S,\leq)$ is a meet semilattice, these are the idempotents which
define the meet operation on $\S$.

\begin{lemma}[Leech]\label{L: Leech} Suppose $\S$ is  an inverse
  monoid such that $(\S,\leq)$ is a meet semilattice. 
For any $s,t\in \S$, $s^\dag t
\wedge 1$ is the smallest idempotent $e$ such that
\begin{equation*} s\wedge t= se= te.
\end{equation*} In particular, $(s\wedge t)^\dag (s\wedge t)= s^\dag
t\wedge 1$.
\end{lemma}

An inverse semigroup $\S$ is \emph{fundamental} if for $s,t\in\S$
\begin{equation*} ses^\dag=tet^\dag\text{ for all }e\in\E(\S)
\end{equation*} only when $s=t$.  Equivalently, $\S$ is fundamental if
the centralizer of $\E(\S)$ in $\S$ is $\E(\S)$.  An inverse semigroup
is \emph{Clifford} if $s^\dag s=ss^\dag$ for all $s\in \S$.
Fundamental and Clifford inverse semigroups play an important role in
the theory of inverse semigroups.  In fact, every inverse semigroup
can be described as the extension of a Clifford inverse semigroup by a
fundamental inverse semigroup.  We explain these concepts now.

Let $\S$ and $\P$ be two inverse semigroups, and let $\pi\colon
\P\rightarrow \E(\S)$ be a surjective homomorphism.  Suppose further
that $\pi|_{\E(\P)}$ is an isomorphism of $\E(\P)$ and
$\E(\S)$.  An inverse semigroup $\G$, together with a surjective
homomorphism $q\colon \G \rightarrow \S$, is an \emph{idempotent
separating extension of $\S$ by $\P$} if there is an embedding $\iota$
of $\P$ into $\G$ such that
\begin{enumerate}
\item $q(g)\in\E(\S)$ if and only if $g=\iota(p)$ for some $p\in\P$;
and
\item $q\circ\iota=\pi$.
\end{enumerate} 
Unless explicitly stated to the contrary, all extensions considered in
the sequel will be idempotent separating.  Thus, we will use the
phrase, `extension of $\S$ by $\P$,' instead of `idempotent separating
extension of $\S$ by $\P$' when discussing extensions.  Also, since
$q\circ \iota=\pi$, we will typically suppress the map $\pi$ and
describe an extension of $\S$ by $\P$ using the diagram,
\[\P\xhookrightarrow{\iota}\G\xrightarrow{q}\S.\]
The extension 
$\P\xhookrightarrow{\iota} \G\xrightarrow{q}\S$ is \textit{a trivial
  extension}
 if there exists a
semigroup homomorphism $j:\S\rightarrow \G$ such that $q\circ
j=\text{id}|_\S$.

We will sometimes identify $\P$ with $\iota(\P)$, so that $\iota$ becomes
the inclusion map.  When this identification is made, we delete 
$\iota$ from the diagram of the extension and simply write
\[\P\hookrightarrow\G\xrightarrow{q}\S.\]

We shall require a notion of equivalent extensions.  The following
definition is a modification of the definitions found
in~\cite{LauschCoInSe} and \cite{LawsonInSe}.    
\begin{definition}\label{gext}
For $i=1,2$ let $\S_i$ and $\P_i$ be inverse semigroups, and suppose
that $\tilde{\alpha}:\S_1\rightarrow \S_2$ and
$\underline{\alpha}:\P_1\rightarrow \P_2$ are fixed isomorphisms of inverse
semigroups. 
The extension
\begin{gather}\label{gext1} 
\P_1 \xhookrightarrow{\iota_1} \G_1\xrightarrow{q_1} \S_1 \\
\intertext{of $\S_1$ by $\P_1$ and the extension}
 \P_2 \xhookrightarrow{\iota_2} \G_2\xrightarrow{q_2} \S_2\label{gext2}
\end{gather} of $\S_2$ by $\P_2$ 
are \emph{$(\underline{\alpha},\tilde{\alpha})$-equivalent}  if there is an isomorphism $\alpha\colon
\G_1\rightarrow \G_2$ such that $q_2\circ
\alpha=\tilde{\alpha}\circ q_1$, and $\underline{\alpha}\circ
\iota_2=\iota_1\circ \alpha.$
\end{definition}
Notice that when the extensions~\eqref{gext1} and~\eqref{gext2} are
$(\underline{\alpha},\tilde{\alpha})$-equivalent, $\tilde{\alpha}\circ q_1\circ
\iota_1=q_2\circ\iota_2\circ\underline{\alpha}$, that is,
\[\tilde{\alpha}\circ\pi_1=\pi_2\circ\underline{\alpha}.\]

\begin{remark} Definition~\ref{gext} differs slightly from
  that given in Lausch~\cite{LauschCoInSe} and
  Lawson~\cite{LawsonInSe}.  These authors assume that $\P_1=\P_2$,
  $\S_1=\S_2$, and both $\tilde{\alpha}$ and $\underline{\alpha}$ are
  the identity maps.  While Definition~\ref{gext} is essentially the
  same as that given by Lausch and Lawson, it 
  enables us to streamline the statements of our main results.
\end{remark}

 In \cite{LauschCoInSe}, Lausch also
  shows that equivalence classes of extensions of inverse semigroups
  may be parametrized by elements of a 2-cohomology group.  Trivial
  extensions as defined above correspond to the neutral element of
  this cohomology group.

Another way to describe extensions of inverse semigroups is via
congruences.  Let $\G$ be an inverse semigroup.  An equivalence
relation $R$ on $\G$ is a \textit{congruence} if it behaves well under
products, that is,
\begin{equation*} (v_1,v_2),(w_1,w_2)\in R \text{ implies
}(v_1w_1,v_2w_2)\in R.
\end{equation*} The quotient of $\G$ by $R$ gives an inverse semigroup
$\S$.  Let $q:\G\rightarrow \S$ denote the quotient map.  Let
\begin{equation}\label{Pdef} \P = \{v\in \G \colon q(v)\in\E(\S) \}.
\end{equation} Then $\P$ is a inverse semigroup, and $\G$ is a
extension of $\S$ by $\P$.

The \textit{Munn congruence} $R_M$ on $\G$ is the congruence,
\[R_M:=\{(v_1,v_2)\in \G\times \G \colon v_1ev_1^\dag=v_2ev_2^\dag
\text{ for all } e\in\E(\G)\}.\] The Munn congruence is the maximal
idempotent separating congruence on $\G$ and the quotient of $\G$ by
$R_M$ is a fundamental inverse semigroup $\S$.  With $\P$ as in~\eqref{Pdef},
$\P$ is a Clifford inverse semigroup, and $\G$ is an idempotent
separating extension of $\S$ by $\P$.

We are interested in inverse monoids with a strong order structure.
Parts (a--c) in the definition below  may be found in Lawson
\cite{LawsonNoGeStDu}.
\begin{definition}\label{D: Boolean inv} An inverse monoid $\S$ with
$0$ is a \textit{Boolean inverse monoid} if
\begin{enumerate}
\item $(\E(\S),\leq)$ is a Boolean algebra;
\item $(\S, \leq)$ is a meet semilattice;
\item if $s,t\in\S$ are orthogonal, their join, $s \vee t$, exists in
$\S$.
\end{enumerate} In addition, we shall say $\S$ is a \textit{locally
complete Boolean inverse monoid} if $\E(\S)$ is a complete Boolean
algebra.  Finally, $\S$ is a \textit{complete Boolean inverse monoid}
if $\S$ satisfies the additional condition,
\begin{enumerate}[resume]
\item for every pairwise orthogonal family $S\subseteq \S$,
$\bigvee_{s\in S} s$ exists in $\S$.
\end{enumerate}
\end{definition}

\begin{remark} A complete Boolean inverse monoid is
necessarily locally complete, see \cite[Corollary~1,
p. 46]{GivantHalmosInBoAl}.
\end{remark}

\begin{example} At first glance, it may appear that local completeness
for a Boolean inverse monoid $\S$ might imply that $\S$ is actually
complete.  Here is an example showing this is not the case.  Let $\H$
be a Hilbert space with orthonormal basis $\{e_j\}_{j\in\bbN}$, and
let $\D$ be the set of all operators $T\in\bh$ for which each $e_j$ is
an eigenvector for $T$.  Let $\S$ be the inverse semigroup generated
by the projections in $\D$ and the rank-one partial isometries,
$\{e_ie_j^*\}_{i,j\in\bbN}$.  Then $\E(\S)$ is a complete Boolean
algebra, and $\{e_{j+1}e_j^*:j\in\bbN\}$ is a pairwise orthogonal family
in $\S$, yet $\bigvee_{j=1}^\infty e_{j+1}e_j^*\notin \S$.
\end{example}

Our main application of Proposition~\ref{P: universal1} is when $\S$
is a Boolean inverse monoid and $\L=\E(\S)$.  For $i=1,2$, let $\S_i$
be Boolean inverse monoids and let $\P_i$ be the inverse semigroup of
partial isometries in $\D_i:=C(\widehat{\E(\S_i)})$.  As in the proof
of Proposition~\ref{P: universal1}, any isomorphism $\theta$ of
$\E(\S_1)$ onto $\E(\S_2)$ uniquely determines a homeomorphism
$\dual{\theta}$ of $\widehat{\E(\S_2)}$ onto $\widehat{\E(\S_1)}$,
which in turn gives a $*$-isomorphism, $\ddual{\theta}$ of $\D_1$ onto
$\D_2$.  Define $\underline{\theta}:= \ddual{\theta}|_{\P_1}$.
Clearly, $\underline{\theta}$ is an isomorphism of $\P_1$ onto $\P_2$.

The map $\underline{\theta}$ allows us to specialize
Definition~\ref{gext} for extensions of Boolean inverse monoids. 

\begin{definition}\label{Tequiv}
For $i=1,2$, let $\S_i$ be Boolean inverse monoids and $\P_i$ be the
partial isometries in $C(\widehat{\E(\S_i)}).$   The extensions
\begin{gather}\label{Tequiv1} 
\P_1 \xhookrightarrow{\iota_1} \G_1\xrightarrow{q_1} \S_1 \\
\intertext{and}
 \P_2 \xhookrightarrow{\iota_2} \G_2\xrightarrow{q_2} \S_2\label{Tequiv2}
\end{gather} 
are \emph{equivalent} 
if there are isomorphisms
$\theta:\S_1\rightarrow \S_2$ and 
$\alpha\colon \G_1\rightarrow \G_2$ such that $q_2\circ
\alpha=\theta\circ q_1$, and
$\iota_2\circ\underline{\theta}=\alpha\circ\iota_1$.
In other
words, these extensions are equivalent 
if there is an isomorphism
$\theta:\S_1\rightarrow\S_2$ such that~\eqref{Tequiv1} is
$(\underline{\theta},\theta)$-equivalent to~\eqref{Tequiv2}.  
\end{definition}

 A \textit{partial homeomorphism} of a topological space
  $X$ is a homeomorphism between two open subsets of $X$. If $s_1$ and
  $s_2$ are partial homeomorphisms, their product $s_1s_2$ has domain
  $\dom(s_1)\cap \ran(s_2)$ and for $x\in X$,
  $(s_1s_2)(x)=s_1(s_2(x))$. 
  In the following proposition, whose
proof is left to the reader, $\O$ denotes the family of clopen subsets
of $\widehat{\E(\S)}$ and $\text{Inv}_\O$ will denote the inverse
semigroup of all partial homeomorphisms of $\widehat{\E(\S)}$ whose
domains and ranges belong to $\O$.
 \begin{proposition}\label{SactionD} Let $\S$ be a Boolean inverse
  monoid and $\D=C(\widehat{\E(\S)})$.  For $s\in \S$, the map
  $\E(\S)\ni e\mapsto s^\dag e s$ determines a partial homeomorphism
  $\bet_s$ of $\widehat{\E(\S)}$ with \[\dom(\beta_s)=
  \{\rho\in\widehat{\E(\S)} : \rho(s^\dag s)=1\}\dstext{and}
  \ran(\beta_s)=\{\rho\in\widehat{\E(\S)}: \rho(ss^\dag)=1\}\] as
  follows: for $e\in\E(\S)$ and $s\in\S$,
\[\beta_s(\rho)(e)=\rho(s^\dag e s).\] 
The map $s\mapsto \beta_s$ is a one-to-one inverse semigroup
homomorphism of $\S$ into the inverse semigroup $\text{Inv}_\O$.  
Moreover, $\beta_s$ determines a
partial action on $\D$:
  for $f\in\D$, define $s^\dag f s\in \D$ by
\[(s^\dag f s)(\rho):=f(\beta_s(\rho)).\]  In particular, when
$e\in\E(\S)$, $s^\dag \chi_{G_e} s= \chi_{G_{s^\dag es}}$.
\end{proposition}

\hyphenation{pre-sently}

\begin{definition}\label{D: Cartan inv} We call an inverse semigroup
$\S$ a \emph{Cartan inverse monoid} if
	\begin{enumerate}
	\item $\S$ is fundamental;
	\item $\S$ is a complete Boolean inverse monoid; and
	\item the character space $\widehat{\E(\S)}$ of the complete
Boolean lattice $\E(\S)$ is a hyperstonean topological space.
	\end{enumerate}
\end{definition}

The choice of name ``Cartan" for these inverse monoids will become
clear presently.  For now we note that condition (c) in
Definition~\ref{D: Cartan inv} tells us that if $\S$ is a Cartan
inverse monoid, then the lattice of idempotents $\E(\S)$ is isomorphic
to the lattice of projections in some abelian von Neumann algebra
\cite[Theorem III.1.18]{TakesakiThOpAlI}.

\begin{remark}  We  emphasize that for two extensions of Cartan
  inverse monoids, equivalence is  always to be taken in the sense of
  Definition~\ref{Tequiv}.
\end{remark}

\begin{remark}  Recall that a 
   \textit{pseudogroup} is an inverse
  semigroup $\S$ of partial homeomorphisms of a topological space $X$.
  By a theorem of V. Vagner~\cite{VagnerOnThAnGp} (or see
  \cite[Section~5.2, Theorem~10]{LawsonInSe}), an inverse semigroup
  $\S$ if fundamental if and only if $\S$ is isomorphic to a
  topologically complete pseudogroup $\T$ consisting of partial
  homeomorphisms of a $T_0$ space $X$; recall that $\T$ is
  \textit{topologically complete} if the family $\{\dom(t): t\in \T\}$
  is a basis for the topology on $X$.

  An application of Vagner's theorem yields a slightly different
  description of Cartan inverse monoids: $\S$ is a Cartan inverse
  monoid if and only if $\S$ is isomorphic to a pseudogroup $\T$ on a
  hyperstonean topological space $X$ such that:
\begin{enumerate}
\item  $\{\dom(t): t\in\T\}=\{E\subseteq
X: E \text{ is clopen}\}$; and 
\item if $\{t_\alpha: \alpha\in
\bbA\}\subseteq \T$ is such that the two families $\{\dom(t_\alpha):
\alpha\in\bbA\}$ and $\{\ran(t_\alpha): \alpha\in\bbA\}$ are each
pairwise disjoint, then there exists $t\in\T$ such that for each
$\alpha\in\bbA$, $t|_{\dom(t_\alpha)}=t_\alpha$. 
\end{enumerate}
Proposition~\ref{SactionD} can be used to produce the
isomorphism. 
\end{remark}

\section{From Cartan MASAs to extensions of inverse
semigroups}\label{S: ext of a pair}

Our goal of this section is to show that every Cartan pair
$(\M,\D)$ uniquely determines an exact sequence of inverse semigroups.
As we will see, these inverse semigroups will be Cartan inverse
monoids.  In Section~\ref{S: Cpair} we show the converse: given an
extension of a Cartan inverse monoid by a natural choice of inverse
semigroup, we can construct a Cartan pair.  Cartan inverse monoids
will play a role analogous to measured equivalence relations of
Feldman-Moore \cite{FeldmanMooreErEqReI,FeldmanMooreErEqReII}.

Let $\M$ be a von Neumann algebra.  Let $\D$ be a MASA (maximal
abelian subalgebra) of $\M$.  The \emph{normalizers of $\D$ in $\M$}
are the elements $x\in\M$ such that
\begin{equation*} x\D x^*\subseteq\D \text{ and } x^*\D x\subseteq\D.
\end{equation*} If a partial isometry $v\in\M$ is a normalizer, then
we call $v$ a \emph{groupoid normalizer}.  The collection of all
groupoid normalizers of $\D$ in $\M$ is denoted by $\G\N(\M,\D)$.  It
is not hard to show that $\G\N(\M,\D)$ is an inverse semigroup with
the adjoint serving as the inverse operation.  The idempotents in the
inverse semigroup $\G\N(\M,\D)$ are the projections in $\D$.

\begin{definition}
A MASA $\D$ in the von Neumann algebra $\M$ is \emph{Cartan} if
\begin{enumerate}
\item there exists a faithful, normal conditional expectation $E$ from
$\M$ onto $\D$;
\item the set of groupoid normalizers $\G\N(\M,\D)$ spans a weak-$*$
dense subset of $\M$.
\end{enumerate} If $\D$ is a Cartan MASA in $\M$, we call the pair $(\M,\D)$
a \emph{Cartan pair}.
\end{definition}

\begin{remark} A MASA $\D$ is usually defined to be Cartan if
it satisfies condition (a) above, and if the unitary groupoid
normalizers of $\D$ in $\M$ span a weak-$*$ dense subset of $\M$.
This is equivalent to the definition given above.   A
proof of the equivalence can be found in \cite[inclusion 2.8, p.\
479]{CameronPittsZarikianBiCaMASAvNAlNoAlMeTh}.
\end{remark}

Let $(\M,\D)$ be a Cartan pair.  Let $\G=\G\N(\M,\D)$ and let
$\P=\G\cap \D$.  Note that $\P$ is the set of all partial isometries
in $\D$.  Thus $\P$ and $\G$ are inverse semigroups with same set of
idempotents, which is the set of projections in $\D$.   That is,
\begin{equation*} \E(\P) = \E(\G) = \textrm{Proj}(\D).
\end{equation*} Let $R_M$ be the Munn congruence on $\G$, let $\S$ be
the quotient of $\G$ by $R_M$, and let
\begin{equation*} q \colon \G \rightarrow \S
\end{equation*} be the quotient map.  It follows that $q|_{\E(G)}$ is
a complete lattice isomorphism from the idempotents of $\G$ onto the
idempotents of $\S$, so
\begin{equation*} \textrm{Proj}(\D)=\E(\P) = \E(\G) \simeq \E(\S).
\end{equation*}
   
\begin{lemma}\label{L: RinP} Let $v\in\G$.  Then $q(v)\in\E(\S)$ if
and only if $v\in\P$.  Thus, $\G$ is an idempotent separating
extension of $\S$ by $\P$.
\end{lemma}

\begin{proof} Suppose that $q(v)\in\E(\S)$.  This means that $v$ is
equivalent to an idempotent $e\in\E(\G)$, that is, $(v,e)\in R_M$ for
some $e\in\E(\G)$.  Since $vIv^\dag = eIe^\dag=e$, for any $f\in\P$ we
have $vf=(vfv^\dag)v=(efe)v=fev=fv$.  It follows that $v$ commutes with
$\D$, and since $\D$ is a MASA in $\M$, we obtain $v\in \P$.

Conversely, if $v\in\P$, then $(v, vv^\dag)\in R_M$.
\end{proof}

\begin{definition} Let $(\M,\D)$ be a Cartan pair.  We call the
extension
\[ \P \hookrightarrow \G\xrightarrow{q} \S\] constructed above the
\textit{extension} for the Cartan pair $(\M,\D)$.
\end{definition}

\begin{proposition}\label{P: S Cartan} Let $\P \hookrightarrow \G\xrightarrow{q} \S$ be the extension for a
  Cartan pair $(\M, \D)$. Then $\S$ is a Cartan
  inverse monoid.
\end{proposition}
\begin{proof} By construction, $\S$ is a fundamental inverse
  semigroup.  Since $\E(\G)$ is the projection lattice of an abelian
  von Neumann algebra, it is a complete Boolean algebra and
  $\widehat{\E(\G)}$ is hyperstonean.  Since $\E(\S)$ is isomorphic to
  $\E(\G)$, $\widehat{\E(\S)}$ is also hyperstonean.

We use \cite[Theorem 1.9]{LeechInMoWiANaSeOr} to show that $\S$ is a
meet semilattice.  Indeed, given $s,t\in \S$, let
\begin{align*} f&=\bigvee\{e\in\E(\S) \colon e\leq st^{\dag}\}\\
&=\bigvee\{e\in\E(\S) \colon e\leq ts^{\dag}\}.
\end{align*} As $\E(\S)$ is a complete lattice, we have that $f$
exists in $\E(\S)$.  We then have $s\wedge t=ft=fs$, so $\S$ is a meet
semi-lattice.

Finally, suppose that $S\subseteq \S$ is a pairwise orthogonal family.
For each $s\in S$, let $v_s\in\G$ satisfy $q(v_s)=s$.  Then $\{v_s:
s\in S\}$ is a family in $\G$ such that for any $s,t\in S$ with $s\neq
t$, $v_s^*v_t=v_sv_t^*=0$.  Then the range projections, $\{v_sv_s^*:
s\in S\}$ are a pairwise orthogonal family; likewise, the source
projections $\{v_s^*v_s: s\in S\}$ are pairwise orthogonal.  Therefore,
$\sum_{s\in S} v_s$ converges strongly in $\M$ to an element $w\in\G$.
Put $r:=q(w)$.  Applying $q$ to each side of the equality, $w
(v_s^*v_s)= v_s$ yields $r\geq s$ for every $s\in S$.  Notice also
that $r^\dag r=\bigvee \{s^\dag s: s\in S\}$.  Hence if $r'\in \S$ and
$r'\geq s $ for every $s\in S$, then $r'{}^\dag r' \geq r^\dag r$. Then
$r=r' (r^\dag r)$, that is, $r'\geq r$.  Thus, $r$ is the least upper
bound for $S$.  It follows that $\S$ is a complete Boolean inverse
monoid.  This completes the proof.

\end{proof}

Our goal now is to show that Cartan pairs uniquely determine their
extensions.

\begin{definition}\label{Miso} For $i=1,2$ let $(\M_i,\D_i)$ be Cartan pairs.  An
\textit{isomorphism} from $(\M_1,\D_1)$ to $(\M_2,\D_2)$ is a
$*$-isomorphism $\theta:\M_1\rightarrow \M_2$ such that
$\theta(\D_1)=\D_2$.
\end{definition}
\begin{remark} For $i=1,2$, let $(X_i,\mu_i)$ be probability
  spaces and suppose $\Gamma_i$ are countable discrete groups acting
  freely and ergodically on $(X_i,\mu_i)$ so that each element of
  $\Gamma_i$ is measure preserving.  Put $\M_i=L^\infty(X_i)\rtimes
  \Gamma_i$ and $\D_i=L^\infty(X_i)\subseteq \M_i$.  Then
  $(\M_i,\D_i)$ are Cartan pairs.  In this context, equivalence in the
  sense of Definition~\ref{Miso} is often called orbit equivalence.
\end{remark}

\begin{theorem}\label{T: same data} For $i=1,2$, suppose $(\M_i,\D_i)$
are Cartan pairs, with associated extensions
\[\P_i\hookrightarrow \G_i\xrightarrow{q_i} \S_i.\] Then $(\M_1,\D_1)$
and $(\M_2,\D_2)$ are isomorphic Cartan pairs if and only if their
associated extensions are equivalent.  Furthermore, when the
extensions are equivalent and $(\M_i,\D_i)$ are in standard form, the
isomorphism is implemented by a unitary operator.
\end{theorem}
\begin{proof}
An isomorphism of Cartan pairs restricts to an
isomorphism of $\G\N(\M_1,\D_1)$ onto $\G\N(\M_2,\D_2)$.  The fact
that the extensions associated to $(\M_1,\D_1)$ and $(\M_2,\D_2)$ are
equivalent follows
easily from their construction.

  Suppose now that the extensions are equivalent.  Let
$\alpha:\G_1\rightarrow\G_2$ and $\tilde{\alpha}:\S_1\rightarrow \S_2$ be
inverse semigroup isomorphisms such that
\[ \tilde{\alpha}\circ q_1= q_2\circ \alpha.\] By examining the image
of $\E(\S_1)$ under $\tilde{\alpha}$, we find that the isomorphism
$\underline{\tilde{\alpha}}$ of $\P_1$ onto $\P_2$ induced by
$\tilde{\alpha}$ is $\alpha|_{\P_1}$.  Thus by Definition~\ref{Tequiv},
$\alpha|_{\P_1}$ is the restriction of a $*$-isomorphism, again called
$\alpha$, of the von Neumann algebra $\D_1$ onto the von Neumann
algebra $\D_2$.

Let $E_i:\M_i\rightarrow\D_i$ be the conditional expectations.  We
claim that
\begin{equation*}\label{Eintertwine} 
(\alpha\circ E_1)|_{\G_1}=E_2\circ \alpha.
\end{equation*} To see this, fix $v\in\G_1$, and 
let $\fJ:=\{d\in\D_1: vd=dv\in\D_1\}$.  Then $\fJ$ is a weak-$*$
closed ideal of $\D_1$.  Therefore, there exists a projection
$e\in\P_1$ such that $\fJ=e\D_1$.  In fact, 
\[e=\bigvee \{f\in \E(\G_1): vf=fv\in\P_1\}.\] Since $E_1(v^*)v$ and
$vE_1(v^*)$ both commute with $\D$, they belong to $\D$; hence
$E_1(v^*)\in\fJ$.  As $\fJ$ is closed under the adjoint operation,
$E_1(v)\in\fJ$.  Therefore, there exists $h\in\D_1$ such that
$E_1(v)=eh$.  It now follows that $E_1(v)=ve$.  Since $\alpha$ is an
isomorphism, we find $\alpha(e) =\bigvee\alpha(\fJ)=\bigvee\{
f_2\in\E(\G_2): \alpha(v)f_2=f_2\alpha(v)\in \P_2\}$.  Hence
$E_2(\alpha(v))=\alpha(v)\alpha(e)$ and the claim holds.

Fix a faithful normal semi-finite weight $\psi_1$ on $\D_1$.  Use
$\alpha$ to move $\psi_1$ to a weight on $\D_2$, that is,
$\psi_2=\psi_1\circ \alpha^{-1}$.  Putting $\phi_i=\psi_i\circ E_i$,
we see $\phi_i$ are faithful semi-finite normal weights on $\M_i$.  Let
$(\pi_i, \fH_i, \eta_i)$ be the associated semi-cyclic representations
(the notation is as in~\cite{TakesakiThOpAlII}) 
and let $\fn_i:=\{x\in \M_i: \phi_i(x^*x)<\infty\}$.  By
\cite[Corollary~1.4.2]{CameronPittsZarikianBiCaMASAvNAlNoAlMeTh},
$\spn(\eta_i(\G_i\cap \fn_i))$ is dense in $\fH_i$.

Let $n\in\bbN$ and suppose $v_1,\dots, v_n\in\G_1\cap \fn_1$ and
$c_1,\dots , c_n\in\bbC$.  Since $(\alpha\circ E_1)|_{\G_1}=E_2\circ
\alpha$, it follows from the definition of $\phi_2$ that
$\alpha(v_j)\in \G_2\cap \fn_2$ and
\begin{align*} \phi_2\left(\left(\sum_{i=1}^n
c_i\alpha(v_i)\right)^*\left(\sum_{i=1}^n c_i\alpha(v_i)\right)\right)
&=\sum_{i,j=1}^n\overline{c_i}c_j\phi_2(\alpha(v_i^*v_j)) \\
&=\sum_{i,j=1}^n\overline{c_i}c_j\phi_1(v_i^*v_j) \\ &=
\phi_1\left(\left(\sum_{i=1}^n c_iv_i\right)^*\left(\sum_{i=1}^n
c_iv_i\right)\right).
\end{align*} Hence the map
\[\eta_1\left(\sum_{i=1}^n c_iv_i\right) \mapsto
\eta_2\left(\sum_{i=1}^n c_i\alpha(v_i) \right)\] extends to a unitary
operator $U:\fH_1\rightarrow \fH_2$.  It is routine to verify that for
$v\in \G_1$, $U\pi_1(v)=\pi_2(\alpha(v))U$.  Therefore the map
$\theta:\M_1\rightarrow \M_2$ given by $\theta(x)=
\pi_2^{-1}(U\pi_1(x)U^*)$ is an isomorphism of $(\M_1,\D_1)$ onto
$(\M_2,\D_2).$
\end{proof}

Let $\M$ be a von Neumann algebra, and let $\D$ be a MASA in $\M$.
Even if $(\M,\D)$ is not a Cartan pair, one can define $\G$, $\P$ and
$\S$ as above to get an extension related to the pair $(\M,\D)$.  The
inverse monoid $\S$ will again be a Cartan inverse monoid.  However,
if $\D$ is not a Cartan MASA in $\M$, the equivalence class of the
extension
\[ \P \hookrightarrow \G\xrightarrow{q} \S\] may arise from a
Cartan pair $(\M_1,\D_1)$ for which $\M$ and $\M_1$ are not
isomorphic.  

\begin{prop} Let $\M$ be a von Neumann algebra and let $\D$ be a MASA
in $\M$.  Then the pair $(\M,\D)$ determines an idempotent separating
exact sequence of inverse semigroups
	\[ \P \hookrightarrow \G\xrightarrow{q} \S,\] where
$\G=\G\N(\M,\D)$, $\P=\G\cap\D$ and $\S$ is a Cartan inverse monoid.
\end{prop}

\section{Representing an extension}\label{S: representation}

The goal of this section is to develop a representation for extensions
of Boolean inverse monoids suitable for the construction of a Cartan
pair from a given extension of a Cartan inverse monoid.  Given an
extension of a Boolean inverse monoid $\S$
\begin{equation*} \P \hookrightarrow \G \xrightarrow{q} \S
\end{equation*} we will ultimately represent $\G$ by partial
isometries acting on a Hilbert space.  This will be achieved after
several intermediate steps, each of which is interesting in its own
right.

A Boolean inverse monoid $\S$ has sufficiently rich order structure to
allow the construction of a representation theory of $\S$ as
isometries on a right $\D$-module constructed from the order
structure; as usual, $\D=C(\widehat{\E(\S)})$.  A key tool in moving from
representations of $\S$ to representations of the extension $\G$ of
$\S$ by $\P$ is the existence of an order preserving sections
$j:\S\rightarrow \G$ which splits the exact sequence of ordered
spaces, $(\P,\leq)\hookrightarrow (\G,\leq)\xrightarrow{q} (\S,\leq)$.
Such sections are discussed in Subsection~\ref{Ss: orderpres}.

In Subsection~\ref{Ss: repo ker} we construct the right Hilbert
$\D$-module $\fA$ mentioned above.  The module $\fA$ will have a
reproducing kernel structure, with the lattice structure of $\S$
represented as $\D$-evaluation maps in $\fA$.

Finally, in Subsection~\ref{Ss: G rep}, we represent $\G$ as partial
isometries in the adjointable operators on $\fA$.  The existence of
the order preserving section plays an important role here.

We alert the reader that because we will we using the theory of right
Hilbert modules, all inner products, either scalar-valued or
$\D$-valued, will be conjugate linear in the first variable.

\subsection{Order preserving sections}\label{Ss: orderpres}

Let $\S$ be a Boolean inverse monoid, let $\P$ be the partial
isometries of $C(\widehat{\E(\S)})$ and let
$\P\hookrightarrow\G\xrightarrow{q}\S$ be an extension.  Since $q$ is
onto, it has a \textit{section}, that is, a map $j:\S\rightarrow \G$
such that $q\circ j=\text{id}|_\S$.  Notice that since $\S$ is
fundamental, $j$ is one-to-one.  Our interest is in those
sections which preserve order.  When the extension is trivial, a
splitting map may be taken to be a semigroup homomorphism, which is
order preserving.  The main result of this subsection is that when
$\S$ is locally complete, then every extension of $\S$ by $\P$,
trivial or not, has an order preserving section.

\begin{definition} Let $\P\hookrightarrow\G\xrightarrow{q}\S$ be an
extension.  We will call a section $j:\S\rightarrow \G$ for $q$ an
\textit{order preserving section for $q$} if
\begin{enumerate}
	\item $j(1)=1$, and
	\item $j(s)\leq j(t)$ for every $s, t\in\S$ with  $s\leq t$.
\end{enumerate}
\end{definition}

\begin{lemma}\label{L: orderpres} Let $j:\S\rightarrow \G$ be a
section for $q$.  The following statements are equivalent. 
\begin{enumerate}
\item The map $j$ is an
order preserving section for $q$.
\item  For every
$e,f\in\E(\S)$ and $s\in\S$,
\begin{equation*} j(esf)=j(e)j(s)j(f).
\end{equation*}
\item For every $s, t\in\S$, $j(s\wedge t)=j(s)\wedge j(t)$ and $j(1)=1$.
\end{enumerate}
\end{lemma}

\begin{proof} $(a)\Rightarrow (b)$. 
Suppose $j$ is an order preserving section.  For any
$e\in\E(\S)$, $e\leq 1$, so $j(e)\leq 1$, hence $j(e)\in\E(\G)$.
Since $q|_{\E(\G)}$ is an isomorphism of $\E(\G)$ onto $\E(\S)$ and
$q\circ j=\text{id}|_\S$, we obtain
$j|_{\E(\S)}=\left(q|_{\E(\G)}\right)^{-1}.$

For $s\in\S$, it follows that
\begin{equation*} j(s)^\dag j(s)=j(s^\dag s),
\end{equation*} because $q(j(s)^\dag j(s))= s^\dag s= q(j(s^\dag s))$.

Now suppose that $s\in\S$ and $e\in\E(\S)$.  Then $j(se)=j(se)
\left(j(se)^\dag j(se)\right)=j(se)j(s^\dag s)j(e).$ Multiply both
sides of this equality on the right by $j(e)$ to obtain
$j(se)=j(se)j(e)$.  Since $se\leq s$, the hypothesis on $j$ gives
\[j(se)=j(se)j(e)\leq j(s)j(e).\] Hence
\[j(se)=j(s)j(e) \left(j(se)^\dag j(se)\right)=j(s)j(s^\dag s)j(e)
=j(s)j(s)^\dag j(s) j(e)=j(s)j(e);\] where the first equality follows
from~\cite[p.~21,~Lemma~6]{LawsonInSe}.

Similar considerations yield $j(es)=j(e)j(s)$ for every $e\in\E(\S)$
and $s\in\S$.  Thus $j(esf)=j(e)j(s)j(f)$.

$(b)\Rightarrow (c)$. 
 Suppose $j(esf)=j(e)j(s)j(f)$
  for all $e,f\in\E(\S)$ and $s\in\S$.  Then for any $e\in\E(\S)$,
  $j(e)\in\E(\G)$.  Since $j$ is a section for $q$, we find $j(1)=1$.
  Notice also that $j|_{\E(\S)}$ is an isomorphism of $(\E(\S),\leq)$
  onto $(\E(\G),\leq)$.  
By Lemma~\ref{L: Leech}, for $s, t\in\S$,
\[ j(s)j(s^\dag t\wedge 1)=j(s\wedge t)=j(t)j(t^\dag s\wedge 1).\]
Therefore, 
$j(s\wedge t)\leq j(s)\wedge j(t).$  To obtain the reverse inequality,
let  $e\in \E(\S)$ be the unique idempotent such that 
$j(e)=(j(s)\wedge j(t))^\dag (j(s)\wedge
j(t))$.  Using~\cite[Lemma~6, p.\ 21]{LawsonInSe}, we find 
\[j(se)=j(s)\wedge j(t)=j(te).\] Since  $j$ is
one-to-one,  $te=se$.  But then $se \leq s\wedge t$.  Applying $j$
to this inequality gives
\[j(s)\wedge j(t)=j(te)\leq j(s\wedge t),\] and~(c) follows.

$(c)\Rightarrow (a)$.
Suppose $s, t\in \S$ with $s\leq t$.
Then $s\wedge t=s$, so $j(s)=j(s\wedge t)=j(s)\wedge j(t)\leq j(t)$.

\end{proof}

It follows immediately that order preserving sections also preserve
the inverse operation.

\begin{corollary}\label{C: orderpres} Let $j$ be an order preserving
section.  Then for all $s\in\S$,
	\begin{equation*} j(s^\dag)=j(s)^\dag.
	\end{equation*}
\end{corollary}

\begin{proof} As
\begin{equation*} j(s)^\dag j(s)=j(s^\dag s),
\end{equation*} it follows that
\begin{equation*} j(s^\dag)^\dag = j(s) ( j(s)^\dag j(s^\dag)^\dag) =
j(s) (j(s^\dag)j(s))^\dag = j(s)j(s^\dag s)= j(s). \qedhere
\end{equation*}
\end{proof}

\begin{remark} Order
preserving sections are implicit in the work of Feldman and Moore.
Indeed if $\S$ is the Cartan inverse monoid consisting of all partial
Borel isomorphisms $\phi$ of the standard Borel space $(X,\B)$
whose graph, $\Gr(\phi)$, is contained in the measured equivalence relation $R$,
then the map $\phi\mapsto \chi_{\Gr(\phi)}$ is an order preserving
section of $\S$ into the inverse semigroup of groupoid normalizers of
the Cartan pair $(M(R,c),D(R,c))$ constructed by Feldman and Moore.
\end{remark}

Note that Lemma~\ref{L: orderpres} and Corollary~\ref{C: orderpres}
hold for extensions of arbitrary Boolean inverse monoids.  We do not
know whether order preserving sections exist in general.  However,
Proposition~\ref{P: order per sec} below shows that when $\S$ is a
locally complete Boolean inverse monoid, such sections always exist.

\begin{lemma}\label{L: idemiso} Let $\S$ be an inverse monoid with
  $0$ such that  
$(\E(\S),\leq)$ is a Boolean algebra.
Let $s\in\S$, let $A_s:=\{t\in\S: t\leq s\}$ and let $B_s:=\{e\in\E(\S):
e\leq s^\dag s\}$.  Then $(A_s,\leq)$ and $(B_s,\leq)$ are Boolean
algebras and the map $\tau_s$ given by $A_s\ni t\mapsto s^\dag t$ is a
complete order isomorphism of $(A_s,\leq)$ onto $(B_s,\leq)$.
\end{lemma}

\begin{proof} The proof is routine after observing that for $e_1,
e_2\in\E(\S)$, $se_1\wedge se_2=s(e_1\wedge e_2)$.
\end{proof}

We recall from Subsection~\ref{Ss: Stone}, that for each $e\in
\E(\S)$, $G_e$ is the clopen set in $\widehat{\E(\S)}$ of characters
supported on $e$.

\begin{proposition}\label{P: order per sec} Let $\S$ be a locally
  complete Boolean inverse monoid, and suppose $\P\hookrightarrow
  \G\xrightarrow{q} \S$ is an extension of $\S$ by $\P$.  Then there
  exists an order preserving section $j:\S\rightarrow \G$ for $q$.
\end{proposition}

\begin{proof} We shall define $j$ in steps, beginning with its
  definition on $\E(\S)$.   
Recall that $q|_{\E(\G)}$ is an isomorphism of $\E(\G)$
onto $\E(\S)$.  Define $j$ on $\E(\S)$ by setting
 $j:=(q|_{\E(\G)})^{-1}$.  

Next, choose a subset
$\B\subseteq \S$ such that,
\begin{enumerate}
\item $1\in\B$;
\item \label{gnormb} if $s_1, s_2\in\B$ and $s_1\neq s_2$, then
$s_1\wedge s_2=0$;
\item $\B$ is maximal with respect to \eqref{gnormb}.
\end{enumerate}  The second condition ensures that
 $\B\cap \E(\S)=\{1\}$.

Define $j$ on the set $\B$ by choosing any 
function  $j:\B\rightarrow \G$  such that $q\circ
j=\text{id}|_\B$ and $j(1)=j(1)$.  For $s\in \B$, Let
$A_s$ and $B_s$ be the sets as in 
Lemma~\ref{L: idemiso}.  Lemma~\ref{L: idemiso} shows that the map 
$A_s\ni t\mapsto s^\dag t$ is an order isomorphism of
$(A_s,\leq)$ onto $(B_s,\leq)$.   Since $\B$ is pairwise meet
orthogonal, the sets $\{A_s: s\in\B\}$ are pairwise disjoint.  
Hence we may extend $j$  from $\B$ to $\bigcup_{s\in\B} A_s$ by
defining $j(se)=j(s)j(e)$ when $s\in\B$ and $e\leq s^\dag s$.  
By construction, $j$ is order-preserving on 
$\bigcup_{s\in\B} A_s$.

We now wish to extend $j$ to the remainder of $\S$.  Let $L=
\S\setminus \left(\bigcup_{s\in\B} A_s\right)$ and let
$\phi:L\rightarrow \G$ be any map such that $q\circ \phi=\text{id}_L$.
We shall perturb $\phi$ so that it becomes order preserving and is
compatible with the map $j$ already defined on $\S\setminus
L=\bigcup_{s\in\B} A_s$.

 Fix $t\in L$ and put $w:=\phi(t)$.  By
definition,  $q(w)=t$.  For each $s\in \B$, let
$h_s:=w^\dag j(t\wedge s)$.  Then $q(h_s)=t^\dag (t\wedge s)= t^\dag t
\wedge t^\dag s\in \E(\S)$, whence $h_s\in \P$.  Set
\[e_s:= q(h_s^\dag h_s).\] Note that $e_s=(s\wedge t)^\dag (s\wedge
t).$ Then $\{e_s: s\in \S\}$ is a pairwise orthogonal subset of
$\{e\in\E(\S): e\leq t^\dag t\}$.  It follows that $\{G_{e_s}: s\in
\B\}$ is a pairwise disjoint family of compact clopen subsets of
$G_{t^\dag t}$.  Also, $G_{t^\dag t}$ is a compact clopen subset of
$\widehat{\E(\S)}$.  The maximality of $\B$ ensures that
$H:=\bigcup_{s\in\B} G_{e_s}$ is a dense open subset of $G_{t^\dag
t}$.  We may thus uniquely define a continuous function
$h:H\rightarrow \bbC$ such that $h|_{G_{e_s}}=h_s$.  Since $\S$ is
locally complete, $G_{t^\dag t}$ is a Stonean space.  By
\cite[Corollary~III.1.8]{TakesakiThOpAlI}, $h$ extends uniquely to a
continuous complex valued function (again called $h$) on all of
$G_{t^\dag t}$.  Extend $h$ to all of $\widehat{\E(\S)}$ by setting
its values to be $0$ on the complement of $G_{t^\dag t}$.  By
construction, $\ran(h)\subseteq \bbT\cup \{0\}$, so $h\in \P$.
Finally, set
\[j(t)=wh.\] The construction shows that for $s\in \B$,
\[j(te_s)=j(t\wedge s)= j(t)j(e_s).\]

For $e\in\E(\S)$ we have (using the facts that $\{r\in\S: r\leq t\}$
is a complete Boolean algebra and $\bigvee\{e_s: s\in\B\}=t^\dag t$),
\begin{align*} j(te)&= j\left( t \bigvee_{s\in\B} ee_s\right)=
j\left(\bigvee_{s\in\B} te_s e\right)=j\left(\bigvee_{s\in\B}(t\wedge
s) e\right) \\
\intertext{(now use Lemma~\ref{L: idemiso} applied to the inverse
  monoid $\G$)}
 &=\bigvee_{s\in\B} j(t\wedge s)j(e)
=\bigvee_{s\in\B}
j(t)j(e_s)j(e)=j(t)\left(\bigvee_{s\in\B}j(ee_s)\right)
=j(t)j(e).
\end{align*} We have now defined $j$ on $L$, so that it preserves
order.  Since $j$ was defined earlier to be order preserving on
$\S\setminus L$, we see that $j$ is order preserving on all of $\S$.
\end{proof}

\subsection{A $\D$-valued reproducing kernel and a right Hilbert
$\D$-module} \label{Ss: repo ker} Let $\S$ be a Boolean inverse monoid
and let $\D=C(\widehat{\E(\S)})$.  As always, we denote the partial
isometries in $\D$ by $\P$.  Let
\begin{equation*} \P \hookrightarrow \G \xrightarrow{q} \S
\end{equation*} be an extension of $\S$ by $\P$.  In this subsection,
we will construct a $\D$-valued reproducing kernel $K:\S\times
\S\rightarrow \D$ whose evaluation functionals $k_s(t):=K(t,s)$
represent the meet-lattice structure of $\S$ in the sense that the
pointwise product of $k_s$ with $k_t$ satisfies $k_sk_t=k_{s\wedge
  t}$.  The completion of $\spn\{k_s\}_{s\in\S}$ yields a $\D$-valued
reproducing kernel right Hilbert $\D$-module, denoted $\fA$.  There is
an action of $\S$ on $\fA$ arising from the left action of $\S$ on
itself: for $s,t\in\S$, $k_t\mapsto k_{st}$.
We modify this action to produce a representation of the extension $\G$ on
$\fA$ by partial isometries in the bounded adjointable maps $\L(\fA)$.
Finally, we obtain a class of representations of $\G$ on a Hilbert
space using the interior tensor product $\fA\otimes_\pi\H$ where
$(\pi,\H)$ is a representation of $\D$ on the Hilbert space 
$\H$.  When $\S$ is a
Cartan inverse monoid and $\pi$ is faithful, 
it is this representation of $\G$ that will
generate a Cartan pair.
 We refer the reader to
\cite{AronszajnThReKe} for more on reproducing kernel Hilbert spaces.

\subsection*{} We begin with the definition of the $\D$-valued
reproducing kernel.   By Proposition~\ref{P: order per sec}, there is an
order preserving section $j\colon \S\rightarrow \G$, which we consider
fixed throughout the
remainder of this section.
\begin{definition} Define $K:\S\times \S\rightarrow \D$ by
\begin{equation*} K(t,s) = j(s^\dag t \wedge 1),
\end{equation*}  and for $s\in \S$, define $k_s:\S\rightarrow \D$ by
\begin{equation*} k_s(t)=K(t,s).
\end{equation*} 
\end{definition}
\begin{remark} By Lemma~\ref{L: Leech}, $K(s,t)$ is the source
  idempotent of $j(s\wedge t)$, that is, $K(s,t)=j((s\wedge
  t)^\dag (s\wedge t))$.
Notice also that 
$K$ is symmetric, that is $K(s,t)=K(t,s)$ for all $s,t\in \S$.
The function $k_s$ should be thought of as the ``$s$-th column'' of
the ``matrix'' $K(t,s)$. 
\end{remark}

We will show in Lemma~\ref{L: Kpos} that $K$ is positive in the sense
that
\begin{equation*} 0\leq \sum_{i,j}^n \overline{c_i}c_j K(s_i,s_j)
\end{equation*} for any finite collection of scalars $c_1,\ldots, c_n$
and $s_1,\ldots,s_n$ in $\S$.  This positivity will allow us define a
$\D$-valued inner product on $\spn\{k_s\}_{s\in\S}$.  The completion
of this span, with respect to the norm from the inner-product, will be
a right Hilbert $\D$-module (Proposition~\ref{P: right D mod}).

The following simple corollary to Lemma~\ref{L: Leech} is immediate, since
$j$ is an order preserving section.

\begin{corollary}\label{C: Leech} For any $s,t\in\S$ and any
$e\in\E(S)$ we have
\begin{equation*} K(t,se)=K(te,s)=K(te,se)=K(t,s)j(e).
\end{equation*} Thus
\begin{equation*} k_{se}=k_sj(e).
\end{equation*}
\end{corollary}

  The significance of $K$ is that an
``integral'' on $\D j(s^\dag s)$, that is, a weight or state on $\D$
restricted to $\D j(s^\dag s)$, may be translated in a consistent way
using $K$ to an ``integral'' on $A_s$.  Remark~\ref{unitaryFM} below
explores this further in the context of the Feldman-Moore
construction.

Corollary~\ref{C: Leech} shows how the map $s\mapsto k_s$
respects the order structure of $\S$.  This is further cemented in the
following lemma, where we show that the mapping $s\mapsto k_s$ is a
meet-lattice representation of $\S$ as a family of functions from $\S$
into the lattice of projections of $\D$.  
Thus we are constructing a
$\D$-module from the order structure of $\S$.

\begin{lemma}\label{L: meet rep} For $r,s,t\in \S$, we have
	\begin{equation}\label{mr0} (s^\dag t \wedge 1)(r^\dag t \wedge
1)=((s\wedge r)^\dag t) \wedge 1,
	\end{equation} hence
	\begin{equation*} K(t,r)K(t,s)=K(t,r\wedge s).
	\end{equation*} In particular, 
	\begin{equation*} k_rk_s=k_{r\wedge s}.
	\end{equation*}
\end{lemma}

\begin{proof} Take any $r,s,t\in \S$.  Applying the isomorphism
  $\tau_t$ of
  $(A_t,\leq)$ onto $(B_t,\leq)$ of 
Lemma~\ref{L: idemiso} and using Lemma~\ref{L: Leech},
 we obtain
\begin{align*}
\left(t^\dag(s\wedge r)\right)\wedge 1&= \tau_t(s\wedge r\wedge t)= 
 \tau_t(s\wedge t)\wedge \tau_t(r\wedge t)\\
&=
 (t^\dag s \wedge 1) (t^\dag r\wedge 1).
\end{align*}
This equality is equivalent to~\eqref{mr0}.
The remaining statements of the lemma
follow.
\end{proof}

We will now show that $K(t,s)=j(s^\dag t\wedge 1)$ defines a positive
$\D$-valued positive kernel.  A key step in this will be to show that
if $\rho$ is a pure state on $\D$, then $\rho\circ K$ defines a positive
semi-definite matrix on $\S$, in the sense of Moore as described in
\cite[p.\ 341]{AronszajnThReKe}.  That is, for each pure state $\rho$ on
$\D$, and for $s_1,\ldots,s_n\in \S$ and $c_1,\ldots,c_n\in \bbC$
\begin{equation*} 0\leq \sum_{i,j}^n \overline{c_i}c_j
K_\rho(s_i,s_j),
\end{equation*} where $K_\rho:=\rho\circ K$. Thus, each pure state $\rho$ 
will determine a reproducing kernel Hilbert space $\H_\rho$ of
functions on $\S$, with kernel $K_\rho$ and point-evaluation functions
$(k_s)_\rho:=\rho\circ k_s$.

\begin{lemma}\label{L: Kpos} Let $n\in\bbN$, let $c_1,\dots, c_n$ be
complex numbers, and let $s_1,\dots , s_n\in\S$.  Then, with respect
to the positive cone in $\D$,
\begin{equation*} 0\leq \sum_{i,j=1}^n \overline{c_i} c_j K(s_i,s_j).
\end{equation*}
\end{lemma}

\begin{proof} Fix a pure state $\rho$ on $\D$.  Note that, as
$\D=C(\widehat{\E(\S)})$, we  view $\rho$ as being in
$\widehat{\E(\S)}$.  With this identification, we have
$\rho(K(s_i,s_j))=\rho(s_j^\dag s_i\wedge 1)$.  Let
$\mathbf{n}=\{1,\dots, n\}$ and
\[ R=\{ (i,j)\in \mathbf{n}\times \mathbf{n}: \rho(K(s_i,s_j))=1\}.\]
We shall show that $R$ is a symmetric and transitive
relation. Symmetry of $R$ is immediate from the fact that
$K(s_i,s_j)=K(s_j,s_i)$.

Suppose for some $i,j,k\in \mathbf{n}$, that
$\rho(K(s_i,s_j))=1=\rho(K(s_j,s_k)).$ To show that $R$ is transitive,
we must show that $\rho(K(s_i,s_k))=1$.  By Lemma~\ref{L: Leech} and
Lemma~\ref{L: meet rep}, we have
\begin{align*} K(s_i,s_j)K(s_j,s_k)&=K(s_j,s_i)K(s_j,s_k)=K(s_j,s_i\wedge
  s_k) \\ &\leq
K(s_i,s_k).
\end{align*} Applying $\rho$ yields $1=\rho(K(s_i,s_k))$. 
Thus $(i,k)\in R$, and $R$ is transitive.

The symmetry and transitivity of $R$ imply that if $(i,j)\in R$, then
both $(i,i)$ and $(j,j)$ belong to $R$.  Let
\[\mathbf{n}_1:=\{i\in\mathbf{n}: \rho(s_i^\dag s_i)=1\}.\] Then
$R\subseteq \mathbf{n}_1\times \mathbf{n}_1$ and $R$ is an equivalence
relation on $\mathbf{n}_1$.

Write $\mathbf{n}_1=\bigcup_{m=1}^r X_m$ as the disjoint union of the
equivalence classes for $R$, and let $T(\rho)\in M_n(\bbC)$ be the
matrix whose $ij$-th entry is $\rho(K(s_j,s_i))=\rho(j(s_i^\dag s_j
\wedge1))$.  Let $\{\xi_j\}_{j=1}^n$ be the standard basis for
$\bbC^n$ and let $\zeta_m=\sum_{j\in X_m} \xi_j$.  Then
\[T(\rho)=\sum_{m=1}^r \zeta_m\zeta_m^*\geq 0,\] where
$\zeta_m\zeta_m^*$ is the rank-one operator, $\xi\mapsto
\innerprod{\zeta_m,\xi}\zeta_m$.  Thus,
\begin{equation*} \rho\left(\sum_{i,j=1}^n \overline{c_i} c_j
K(s_i,s_j)\right)= \innerprod{T(\rho)\begin{pmatrix}c_1\\ \vdots \\
c_n\end{pmatrix}, \begin{pmatrix}c_1\\ \vdots \\ c_n\end{pmatrix}}\geq
0.
\end{equation*} As this holds for every pure state $\rho$ on ${\D}$, we
find that $\sum_{i,j=1}^n \overline{c_i} c_j K(s_i,s_j)\geq 0$ in
$\D$.
\end{proof}

Let
\begin{equation*} \fA_0 = \spn\{k_s \colon s\in \S\}.
\end{equation*} Our goal is show that there is a natural $\D$-valued
inner product on $\fA_0$.  Completing with respect to the norm induced
by this inner product will yield a Banach space $\fA$ of functions from $\S$
into $\D$.  For a function $u\in\fA$ and $d\in\D$,
define $u d$ to be the function from $\S$ into $\D$ given by
\[(ud)(s):= u(s) d.\]  We shall show that this  makes $\fA$ into a
 right Hilbert $\D$-module.

In order to study the space $\fA_0$, we use the 
reproducing kernel Hilbert spaces generated by composing elements of 
$\fA_0$ with
a given pure state $\rho$ on $\D$.

For each pure state $\rho$ on $\D$ and $f\in\fA_0$, write $f_\rho$ for the
function $\rho\circ f$ on $\S$.  In particular, notice that for
$s\in\S$, $(k_s)_\rho(t)=\rho(K(s,t)).$ Likewise, set
$K_\rho(t,s):=\rho(K(t,s))$. Finally, put
\[\fA_\rho:=\spn\{(k_s)_\rho : s\in\S\}.\]

It was shown in Lemma~\ref{L: Kpos} that $K_\rho$ is a positive matrix
in the sense of~\cite{AronszajnThReKe}.  Thus $K_\rho$ determines a
reproducing kernel Hilbert space $\H_\rho$, with $\fA_\rho$ dense in
$\H_\rho$.  For $u_\rho, v_\rho\in\fA_\rho$, we may find $n\in\bbN$,
$c_i, d_i\in \bbC$, and $s_i, t_j\in\S$ so that $u_\rho=\sum_{i=1}^n
c_i (k_{s_i})_\rho$ and $v_\rho= \sum_{i=1}^n d_i (k_{t_i})_\rho$.
The formula,
\begin{align*}
\innerprod{u_\rho,v_\rho}_{\rho} &:=\innerprod{\sum_{i=1}^n c_i
(k_{s_i})_\rho, \sum_{j=1}^n d_i (k_{t_j})_\rho}_\rho
= \sum_{i,j=1}^n
\overline{c_i} d_j K_\rho(s_i,t_j)\\ &=\sum_{i,j=1}^n \overline{c_i} d_j
\rho(K(s_i,t_j))
\end{align*}
gives a well-defined inner product on $\fA_\rho$ by
the Moore-Aronszajn Theorem (see \cite[paragraph (4),
p. 344]{AronszajnThReKe}).  In particular,
$\norm{u_\rho}:=\innerprod{u_\rho,u_\rho}_\rho$ is a norm on
$\fA_\rho$.

We now are ready to define a $\D$-valued inner product on $\fA_0$.  We
wish the inner product to be conjugate linear in the first variable
and to satisfy,
\[\innerprod{k_s,k_t}=k_t(s)=K(s,t).\] As before, for $u, v\in\fA_0$,
write $u$ and $v$ as linear combinations, $u=\sum_{i=1}^n c_i k_{s_i}$
and $v= \sum_{i=1}^n d_i k_{t_i}$ of elements from $\{k_s: s\in\S\}$.
Then for a pure state $\rho$ on $\D$,
\[ \rho\left(\sum_{i,j=1}^n \overline{c_i}d_j K(s_i,t_j)\right)
=\innerprod{u_\rho,v_\rho}_\rho.\] Hence $\sum_{i,j=1}^n
\overline{c_i}d_j K(s_i,t_j)\in\D$ depends only upon $u$ and $v$ and
not on the choice of linear combinations from $\{k_s: s\in\S\}$ which
represent them.  Therefore, the following definition makes sense.
\begin{definition}\label{D: inner prod} For $u=\sum_{i=1}^n c_i
k_{s_i}$ and $v= \sum_{i=1}^n d_i k_{t_i}$ in $\fA_0$,
\[ \innerprod{u,v}:=\sum_{i,j=1}^n \overline{c_i}d_j K(s_i,t_j)\] is a
well-defined $\D$-valued inner product on $\fA_0$ which is
sesquilinear in the first variable.
\end{definition}

Finally, define $\fA$ to be the completion of $\fA_0$ relative to the
norm,
\[\norm{u}:=\norm{\innerprod{u,u}}_\D^{1/2}.\]

\begin{proposition}\label{P: right D mod} The space $\fA$ is a Banach
space of functions from $\S$ to $\D$, satisfying
\begin{equation*} u(s)=\<k_s,u\>
\end{equation*} for each $u\in\fA$ and $s\in\S$.

The right action of $\D$ on $\fA$ given by
\begin{equation*} (vh)(s)=v(s)h
\end{equation*} for  $h\in\D$, $v\in\fA$ and $s\in S$, makes $\fA$
into a right Hilbert $\D$-module.
\end{proposition}

\begin{proof} It follows from Definition~\ref{D: inner prod} that, for
$u\in\fA_0$ and $s\in\S$, $u(s)=\<k_s,u\>$.  For any $u,v\in\fA_0$ and
any pure state $\rho$ on $\D$ we have
\begin{equation*} \rho(\<u,v\>)=\<u_\rho,v_\rho\>_\rho.
\end{equation*} Hence
\begin{equation*} \<u,v\>^*\<u,v\>\leq\<u,u\>\<v,v\>
\end{equation*} for any $u,v\in \fA_0$.  For $s\in \S$, it now follows that the
evaluation map 
$\varepsilon_s \colon  \fA_0 \rightarrow \D$ given by 
\[\fA_0\ni u\mapsto u(s)\in \D\]
is continuous.  Thus, the evaluation map $\varepsilon_s$ extends to a
continuous map on $\fA$.  Therefore, $\fA$ is a Banach space of
functions from $\S$ into $\D$, and for each $u\in\fA$ and $s\in\S$ we
have $u(s)=\<k_s,u\>$.

Write $\D_0$ for the linear span of the projections in $\D$.  Note
that, for $u,v\in\fA_0$ we have that $\<u,v\>\in\D_0$.
Further if $h\in\D_0$, then for any
$v\in\fA_0$, $vh\in\fA_0$.  To see this, 
write $v$ and $h$ as linear combinations,
$v=\sum_{n=1}^N \alpha_nk_{t_n}$ and $h=\sum_{m=1}^N \beta_m j(e_m)$,
where $\alpha_n, \beta_m\in \bbC$, $t_n\in \S$ and $e_m\in\E(\S)$.  
Then, by Corollary~\ref{C: Leech}, we have
\[vh=\sum_{n,m=1}^N \alpha_n\beta_m k_{t_n}j(e_m)=\sum_{n,m=1}^N
\alpha_n\beta_m k_{t_ne_m}\in\fA_0.\]

A calculation, again using Corollary~\ref{C: Leech}, shows that for
$u\in\fA_0$, we have
\begin{equation}\label{Eq: ract} \innerprod{u,vh}=\innerprod{u,v}h.
\end{equation}

Continuity of multiplication by $h$  
then yields $vh\in\fA$ for every $v\in\fA$ and $h\in\D_0$.  As the
projections of $\D$ span a norm dense subset of $\D$, we then obtain
$vh\in\fA$ for every $v\in\fA$ and $h\in\D$.  Since equation~\eqref{Eq: ract}
passes to the completions, the proof is complete.
\end{proof}

\subsection{The construction of the representation}\label{Ss: G rep}

Our goal in this subsection is to construct a representation of an
extension $\G$ of a Boolean inverse monoid $\S$ by $\P$, where $\P$ is
the semigroup of partial isometries in $\D=C(\widehat{\E(\S)})$.  We
will represent $\G$ in the adjointable operators on $\fA$, where $\fA$
is the right Hilbert $\D$-module constructed in the previous
subsection.

We fix an extension
\begin{equation*} \P \hookrightarrow \G \xrightarrow{q} \S.
\end{equation*} By Proposition~\ref{P: order per sec}, there is an
order preserving section $j \colon \S \rightarrow \G$, such that
$q\circ j = \textrm{id}_\S$.  Thus, while the extension need not be
a trivial trivial extension of inverse semigroups,  we do
 have a  splitting at the level of partially ordered
sets:
\[\xymatrix{ (\P,\leq)\ar@{^{(}->}[r] &(\G,\leq)\ar@/^/[r]^q
&(\S,\leq) \ar@/^/[l]^j}.\]

A construction of Lausch \cite{LauschCoInSe} shows that, up to an equivalent extension, $\G$ can
be viewed as the set
\begin{equation*} \{ [s,p] \colon s\in \S,\ p\in\P,\ q(p^\dag
p)=s^\dag s \}.
\end{equation*} That is, every element $v=[s,p]\in \G$ consists of a
function $p\in\P$ ``supported" on an element $s\in\S$.  The product on
$\G$ is then determined by a cocycle function
\begin{align*} \alpha \colon & \S \times \S \rightarrow \P \\ & [s,t]
\mapsto j(st)^\dag j(s) j(t).
\end{align*} We do not wish to dwell on this viewpoint.  It can be
computationally cumbersome, and while it lies behind much of our work,
most of it is unnecessary for our
constructions.  However, we will need to be able to describe certain
elements of $\G$ in terms of their support in $\S$ and a function in
$\P$.  In order to do this, we construct a cocycle-like function,
related to the cocycle $\alpha$.

\begin{definition}\label{D: cocycle like} Define a cocycle-like
function $\sigma \colon \G \times \S \rightarrow \P$
by
\begin{equation*} \sigma(v,s)=j(q(v)s)^\dag vj(s)=j(s^\dag
q(v^\dag))vj(s).
\end{equation*}
Since
\begin{equation*} q(\sigma(v,s))=s^\dag q(v^\dag v)s\in \E(\S),
\end{equation*} $\sigma(v,s)\in\P$.  Thus $\sigma$ indeed maps
$\G\times \S$ into $\P$. 
\end{definition}

\begin{remark} Lausch's cocycle $\alpha$ can be recovered from $\sigma$
as follows.  For any $s,t\in\S$ we have
\begin{equation*} \sigma(j(s),t)=j(st)^\dag j(s)j(t)=\alpha(s,t).
\end{equation*} 
Also, observe that for all $v\in\G$ and $s\in\S$ we have
\begin{equation*} vj(s) = j(q(v)s)\sigma(v,s).
\end{equation*} Thus elements of the form $vj(s)$ can be described in
terms of an element $q(v)s\in \S$ and $\sigma(v,s)\in\P$.  The left action
of $\G$ on the set  $j(\S)$ will be used to construct the representation.
\end{remark}

\begin{lemma}\label{chop}  Let $N\in \bbN$ and let $s_1,\dots,
  s_N$ be non-zero elements of $\S$.  There exists a finite set
  $A\subseteq \S$ with the following properties.
\begin{enumerate}
\item $0\notin A$.
\item The elements of $A$ are pairwise meet orthogonal.
\item Each $a\in A$ satisfies:
\begin{enumerate}
\item[i)]  for $1\leq n\leq N$, $a\wedge s_n\in \{a,0\}$; and
\item[ii)] there exists $1\leq n\leq N$ such that
  $a\wedge s_n=a$.
\end{enumerate}
\item For each $1\leq n\leq N$, $s_n=\bigvee\{a\in A: a\leq s_n\}$.
\end{enumerate}
\end{lemma}
\begin{proof}  Throughout the proof, when $s, t\in\S$,
  we use $s\setminus t$ to denote the element $s(s^\dag s \meet \neg
  (t^\dag t))$; thus $s\setminus t$ is orthogonal to $s\meet t$ and 
satisfies $(s\setminus t)\vee (s\meet t)= s$.

  We argue by induction on $N$.  If $N=1$, take $A=\{s_1\}$.  Suppose
  now that $N>1$ and the result holds whenever we are given non-zero
  elements $s_1,\dots, s_{N-1}$ of $\S$.

  Let $s_1,\dots, s_N$ be non-zero elements of $\S$ and let $A_{N-1}$
  be the set constructed using the induction hypothesis applied to
  $s_1,\dots , s_{N-1}$.  For each $b\in A_{N-1}$, let $C_b:=\{b\meet
  s_N, b\setminus s_N\}\setminus \{0\}$ and 
  put $X:=\bigcup_{b\in A_{N-1}} C_b$.  Since the elements of
    $A_{N-1}$ are pairwise meet disjoint, so are the elements of $X$.
    Let $t:= \bigvee\{x\in X: x\leq s_N\}$ and let  
$r:=s_N\setminus t$.  Notice that $r\meet x=0$ for any $x\in X$.  
 Finally, define
\[A:=\begin{cases} \{r\}\cup X&\text{ if $r\neq 0$};\\ 
X&\text{ if $r=0$.}
\end{cases} \] Then $A$ is pairwise meet orthogonal, and $0\notin A$. 
 By construction, we have $s_N=\bigvee\{a\in A: a\leq
s_N\}$.  Moreover, if $1\leq n\leq N-1$ and $b\in A_{N-1}$ with $b\leq
s_n$, then $b=\bigvee C_b$.  Since $A_{N-1}$ satisfies property (d),
 we 
obtain $s_n=\bigvee\{a\in A: a\leq s_n\}$.  Thus $A$ satisfies
property (d) also.

Property (c) is equivalent to the statement that for $a\in A$, 
\[\{0\}\neq \{a\wedge s_n: 1\leq n\leq N\}\subseteq \{0,a\}.\]
For 
$a\in X$, this clearly holds.
Suppose $r\neq 0$. Then $r\wedge s_N=r$, so $\{0\}\neq \{r\wedge
s_n: 1\leq n\leq N\}$.  If $1\leq n\leq N-1$, then $r\wedge s_n=
r\wedge (\bigvee\{b\in X: b\leq s_n\})=0$.  Hence 
$\{r\wedge s_n: 1\leq n\leq N\}\subseteq \{r,0\}$.
  Therefore $A$ satisfies
the requisite properties and we are done.
\end{proof}

We now have all the ingredients we need to construct our
representation of $\G$.  We recall, that $\fA$ is the the right
Hilbert $\D$-module constructed in Subsection~\ref{Ss: repo ker}.

\begin{theorem}\label{T: Grepdef} For $v\in \G$ and $s\in\S$, the
formula,
\[\lambda(v)k_s:= k_{q(v)s} \cocycle(v,s)\] determines a partial
isometry $\lambda(v)\in\L(\fA)$.  Moreover, $\lambda:\G\rightarrow
\L(\fA)$ is a one-to-one representation of $\G$ as partial isometries
in $\L(\fA)$.
\end{theorem}

\begin{proof} Fix $v\in\G$, and set $r:=q(v)$.  Given $s_1,\dots,
s_N\in\S$, let $A$ be the set constructed
in Lemma~\ref{chop}.   Choose $c_1,\dots, c_N\in\bbC$.

For $a\in A$ and $1\leq m\leq N$, put
\[A_m:=\{b\in A: b\leq s_m\}\dstext{and} c_a:=\sum\{c_n: a\leq
s_n\}.\] Notice that the elements of $A_m$ are pairwise orthogonal,
and $\bigvee A_m=s_m.$

We first note that
\begin{equation}\label{Eq: Grepdef1} \sum_{n=1}^N c_n
k_{s_n}=\sum_{a\in A} c_a k_{a}.
\end{equation} To see this, for any $t\in\S$,
\begin{equation*} K(t,s_n)=\sum_{a\in A_n} K(t,a).
\end{equation*} Thus,
\begin{align*} \left(\sum_{n=1}^N c_n k_{s_n}\right)(t)&= \sum_{n=1}^N
c_nK(t,s_n) = \sum_{n=1}^N c_n \ \sum_{a\in A_n} K(t,a)\\ &= \sum_{a\in
A} c_a K(t,a) = \left(\sum_{a\in A} c_a k_{a}\right)(t).
\end{align*}

Secondly, we claim
\begin{equation}\label{Eq: Grepdef2} \sum_{n=1}^N c_n
k_{rs_n}\cocycle(v,s_n)=\sum_{a\in A} c_a k_{ra}\cocycle(v,a).
\end{equation} To see this, first notice that if $a\in A$ and $a\leq
s_n$, then using the fact that $j$ is order preserving we have
\begin{equation*} \sigma(v,s_n)K(t,ra)=\sigma(v,s_n)j(a^\dag
a)K(t,ra)=\sigma(v,a)K(t,ra).
\end{equation*} Thus,
\begin{align*} \left(\sum_{n=1}^N c_n
k_{rs_n}\cocycle(v,s_n)\right)(t)&= \sum_{n=1}^N c_n \sigma(v,s_n)
K(t,rs_n) \\ &= \sum_{n=1}^N c_n \left( \sum_{a\in A_n}
\sigma(v,s_n)K(t,ra)\right)\\ &= \sum_{n=1}^N c_n \left( \sum_{a\in
A_n} \sigma(v,a)K(t,ra)\right)\\ &= \sum_{a\in A} c_a
\sigma(v,a)K(t,ra)\\ &=\left(\sum_{a\in A}
c_ak_a\cocycle(v,a)\right)(t),
\end{align*} as desired.

Notice that if $a,b\in A$ are distinct, then
$\innerprod{k_{ra}\cocycle(v,a),
k_{rb}\cocycle(v,b)}=0=\innerprod{k_a,k_b}.$ Thus, using ~\eqref{Eq:
Grepdef2}, then~\eqref{Eq: Grepdef1},
\begin{align*} \innerprod{\sum_{n=1}^N c_n
k_{rs_n}\cocycle(v,s_n),\sum_{n=1}^N c_n
k_{rs_n}\cocycle(v,s_n)}&=\innerprod{\sum_{a\in A} c_a
k_{ra}\cocycle(v,a), \sum_{a\in A} c_a k_{ra}\cocycle(v,a)}\\
&=\sum_{a\in A} |c_a|^2j(a^\dag r^\dag r a)\\ &= \innerprod{\sum_{a\in A}
c_a k_{ra},\sum_{a\in A} c_a k_{ra}}\\ &\leq \sum_{a\in A}
|c_a|^2j(a^\dag a)\\ &= \innerprod{\sum_{a\in A} c_a k_{a},\sum_{a\in
A} c_a k_{a}}\\ &=\innerprod{\sum_{n=1}^N c_n k_{s_n},\sum_{n=1}^N c_n
k_{s_n}}.
\end{align*} Therefore,
\begin{equation*} \norm{\sum_{n=1}^N c_n \lambda(v)k_{s_n}}\leq
\norm{\sum_{n=1}^n c_n k_{s_n}}.
\end{equation*} It follows that we may extend $\lambda(v)$ linearly to
a contractive operator from $\fA_0$ into $\fA$.  Finally extend
$\lambda(v)$ by continuity to a contraction in $\B(\fA)$, the bounded
operators on $\fA$.

We next show that $\lambda(v)$ is adjointable.  Note that, for $s,
t\in \S$, it follows from Lemma~\ref{L: Leech} and Corollary~\ref{C:
Leech} that
\begin{equation*} K(t,rs)=K(r^\dag t,s).
\end{equation*} Further, setting $f=(s^\dag r^\dag t)\wedge 1=j^{-1}(K(t,rs))$, it
follows from Lemma~\ref{L: Leech} that
\begin{equation*} rs\wedge t=rsf=tf
\end{equation*} and
\begin{equation*} s\wedge r^\dag t= sf=r^\dag t f.
\end{equation*} Hence
\begin{equation*} \sigma(v,s)^\dag K(rs,t)=\sigma(v^\dag,t)K(s,r^\dag
t).
\end{equation*} Therefore for any $s,t\in\S$,
\begin{align*} \innerprod{\lambda(v)k_s, k_t} &=
\innerprod{k_{rs}\cocycle(v,s), k_t} =\cocycle(v,s)^\dag K(rs,t) \\ &=
\cocycle(v^\dag, t) K(s,r^\dag t) = \<k_s, \lambda(v^\dag) k_t\>.
\end{align*} This equality implies that $\lambda(v)$ is adjointable
and $\lambda(v)^*=\lambda(v^\dag).$

We now show that $\lambda$ is a homomorphism.  Suppose that
$v_1,v_2\in\G$ and $s\in\S$.  Then
\begin{align*} \lambda(v_1)(\lambda(v_2)k_s)&=
\lambda(v_1)(k_{q(v_2)s} \cocycle(v_2,s))\\ &=
(\lambda(v_1)k_{q(v_2)s}) \cocycle(v_2,s)\\ &= k_{q(v_1v_2)s}\,\,
\cocycle(v_1,q(v_2)s)\,\, \cocycle(v_2,s).
\end{align*} But
\begin{align*} \cocycle(v_1,q(v_2)s)\,\, \cocycle(v_2,s)& =
j(q(v_1)q(v_2)s))^\dag v_1j(q(v_2)s) \,\, j(q(v_2)s)^\dag v_2j(s)\\ &=
j(q(v_1v_2)s))^\dag v_1j(q(v_2)s) \,\, j(s^\dag q(v_2)^\dag) v_2j(s)\\
&= j(q(v_1v_2)s))^\dag v_1 (v_2j(ss^\dag) v_2^\dag) v_2j(s)\\ &=
j(q(v_1v_2)s))^\dag v_1 v_2 v_2^\dag v_2 j(ss^\dag) j(s)\\
&=j(q(v_1v_2)s))^\dag v_1v_2j(s)=\cocycle(v_1v_2,s).
\end{align*} Hence $\lambda(v_1)\lambda(v_2)k_s=\lambda(v_1v_2)k_s$.
As $\spn\{k_s: s\in\S\}$ is dense in $\fA$, we conclude that
$\lambda(v_1v_2)=\lambda(v_1)\lambda(v_2).$

It follows that for every $e\in\E(\G)$, $\lambda(e)$ is a projection.
Furthermore, for $v\in \G$, $\lambda(v)$ is a partial isometry because
$\lambda(v)^*=\lambda(v^\dag)$.

It remains to show that $\lambda$ is one-to-one.  We first show that
$\lambda|_{\E(\G)}$ is one-to-one.  So suppose $e, f\in\E(\S)$ and
$\lambda(j(e))=\lambda(j(f))$.  Then for every $s\in\S$,
$k_{es}\cocycle(j(e),s)=k_{fs}\cocycle(j(f),s)$, whence
$k_{es}j(s^\dag e s)=k_{fs}j(s^\dag fs)$.  Taking $s=1$ gives
$k_ej(e)=k_fj(f)$.  Evaluating these functions at $t=1$ gives
$j(e)=j(f)$, so $\lambda|_{\E(\G)}$ is one-to-one.

Now suppose $v_1,v_2\in\G$ and $\lambda(v_1)=\lambda(v_2)$.  Then
\[\lambda(v_1^\dag v_1)=\lambda(v_1^\dag v_2)=\lambda(v_1^\dag
v_2)^*=\lambda (v_2^\dag v_1)=\lambda(v_2^\dag v_2).\] Likewise,
\[\lambda(v_1v_1^\dag) =\lambda(v_1v_2^\dag)
=\lambda(v_2v_1^\dag)=\lambda(v_2v_2^\dag).\] Hence $v_1^\dag
v_1=v_2^\dag v_2$ and $v_1v_1^\dag =v_2v_2^\dag$.  For any
$e\in\E(\S)$, we have
\begin{align*}
\lambda(v_1 j(e) v_1^\dag)&= \lambda(v_1 v_1^\dag v_1 j(e) v_1^\dag
v_1 v_1^\dag) = \lambda(v_1 v_2^\dag v_2 j(e) v_2^\dag v_2 v_1^\dag)
=\lambda(v_2v_2^\dag v_2 j(e) v_2^\dag v_2 v_2^\dag)\\
&=\lambda(v_2j(e)
v_2^\dag).
\end{align*}
Hence $v_1j(e)v_1^\dag =v_2j(e)v_2^\dag$.  Since this
holds for every $e\in\E(\S)$ and $\S$ is fundamental, we conclude that
\[q(v_1)=q(v_2).\]

 Put $e:=q(v_1^\dag v_1)$ and $s:=q(v_1)$.  Since the functions
$\lambda(v_1)k_e$ and $\lambda(v_2)k_e$ agree, we obtain, $k_{s}
j(s)^\dag v_1=k_{s} j(s)^\dag v_2 $.  Evaluating these functions at
$t=s$ gives, $j(s)^\dag v_1=j(s)^\dag v_2$.  Now multiply each side of
this equality on the left by $j(s)$ to obtain $v_1=v_2$.
\end{proof}

We recall some facts about interior tensor products which may be found
in~\cite[pages 38--44]{LanceHiC*Mo}.  We will only need the interior
tensor product of $\fA$ with a Hilbert space.  If $\H$ is a Hilbert
space and $\pi:\D\rightarrow \H$ is a $*$-representation, the balanced
tensor product of $\fA\otimes_\D \H$ is the quotient of the algebraic
tensor product of $\fA$ with $\H$ by the subspace generated by
$\{ud\otimes \xi -u\otimes \pi(d)\xi: u\in \fA, d\in \D, \xi\in\H\}$.
The balanced tensor product admits a semi-inner product given by
\[\innerprod{u_1\otimes \xi_1,u_2\otimes \xi_2}= \innerprod{\xi_1,
\pi(\innerprod{u_1,u_2})\xi_2}.\] Let $N=\{x\in \fA\otimes_\D \H:
\innerprod{x,x}=0\}$.  The completion of $(\fA\otimes_\D \H)/N$ yields
the \textit{interior tensor product} of $\fA$ with $\H$, and is
denoted $\fA\otimes_\pi \H$.  Notice this is a Hilbert space.

Recall also that there is a $*$-representation $\pi_*:
\L(\fA)\rightarrow \B(\fA\otimes_\pi \H)$ given by
\begin{equation}\label{inducedpi} \pi_*(T)(u\otimes \xi)=(Tu)\otimes
\xi.
\end{equation} This representation is strictly continuous on the unit
ball of $\L(\fA)$ and is faithful whenever $\pi$ is faithful~\cite[p.\
42]{LanceHiC*Mo}.
The following is now immediate.

\begin{corollary}\label{C: HilbertSpRep} Let let $\pi:\D\rightarrow
  \bh$ be a $*$-representation of $\D$ on the Hilbert space $\H$.
  Then $\lambda_\pi:= \pi_*\circ \lambda$ is a representation of $\G$
  by partial isometries on $\fA\otimes_\pi \H$.  If $\pi$ is faithful,
  then $\lambda_\pi$ is one-to-one.
\end{corollary}

\begin{remark}\label{unitaryFM}  In this remark, we continue 
  to outline a comparison of our constructions with those of Feldman
  and Moore.  Full details are left to the interested reader.

  Assume that $(X,\B)$ is a standard Borel space, $R\subseteq X\times
  X$ is a Borel equivalence relation (with countable equivalence
  classes), $\mu$ is a quasi-invariant measure on $X$, $\S$ is the
  Cartan inverse monoid of all partial Borel isomorphisms on $X$ whose
  graphs are contained in $R$, and $\nu$ is right-counting measure on
  $R$ (see~\cite[Theorem~2]{FeldmanMooreErEqReI}).  Let $c$ be a
  2-cocycle on the equivalence relation $R$.  As
  in~\cite[Section~2]{FeldmanMooreErEqReII}, we assume that $c$ is
  normalized (and hence skew-symmetric) in the sense
  of~\cite[page~314]{FeldmanMooreErEqReI}.  Using the Feldman-Moore
  construction (c.f.~\cite[Section~2]{FeldmanMooreErEqReII}), form the
  Cartan pair $(M(R,c), D(R,c))$. Recall that $M(R,c)$ consists of
  certain measurable functions on $R$ and that $D(R,c)$ are those
  which are supported on the diagonal $\{(x,x): x\in X\}$ of $R$.
  Note that $D(R,c)\simeq L^\infty(X,\mu)$.

  As done in Section~\ref{S: ext of a pair}, let $\G=\G\N(M(R,c),
  D(R,c))$ and let $\P$ be the partial isometries in $D(R,c)$.  For
  $v\in \G$, the map \[D(R,c) v^*v\ni dv^*v \mapsto vdv^*\in D(R,c)
  vv^*\] is an isomorphism of $D(R,c) v^*v$ onto $D(R,c) vv^*$ and
  hence arises from a partial Borel isomorphism of $X$.  This partial
  Borel isomorphism is $q(v)$.  Finally, let $j:\S\rightarrow \G$ be
  given by $j(s):=\chi_{\Gr(s)}$.  We have now explicitly described
  the the various components of the extension,
\[\xymatrix{ \P\ar@{^{(}->}[r] &\G\ar@/^/[r]^q
&\S\ar@/^/[l]^j}\] 
and the section $j$ associated with a Cartan pair
  $(M(R,c),D(R,c))$ presented using the Feldman-Moore construction.
 
  Next, we give a formula for the ``cocycle-like'' function of
  Definition~\ref{D: cocycle like} in terms of the Feldman-Moore data.
  For $v\in\G$, we know $h_v:=j(q(v))^\dag v\in D(R,c)$ and
  $v=j(q(v)) h_v$.  Using the fact that $c$ is a normalized cocycle,
  for almost all $(x,y)\in R$ we obtain
\[v(x,y)=\chi_{\Gr(q(v))}(x,y) h_v(y,y).\]  Now for $s\in \S$,
$\sigma(v,s)=j(q(v)s)^\dag v j(s)$.  A computation then shows that
 for $(x,y)\in R$,
\begin{align}
(vj(s))(x,y)&=\chi_{\Gr(q(v)s)}(x,y) \,\, h_v(s(y),s(y)) \,\,
c((q(v)s)(y), s(y),y),  \notag\\ 
\intertext{and, (again using the fact that $c$ is normalized)}
\sigma(v,s)(x,y)&
= \chi_{\Gr(s^\dag q(v^\dag v) s)}(x,y)\,\, 
h_v(s(y),s(y)) \,\, c((q(v)s)(y), s(y),y).\label{sigmaformula}
\end{align}

  Let $\pi$ be the representation of $D(R,c)$ on 
   $\H:=L^2(X,\mu)$ as multiplication operators: for $f\in D(R,c)$,
  $\xi\in L^2(X,\mu)$ and $x\in X$, $(\pi(f)\xi)(x)=f(x,x)\xi(x)$.
  Clearly $\pi$ is a faithful, normal representation of $D(R,c)$.
  
Our next task is to 
  observe that the representation $\lambda_\pi$ of $\G$ on
  $\B(\fA\otimes_\pi \H)$ is
  unitarily equivalent to the identity representation of $\G$ on $L^2(R,\nu)$. 

  For $s\in\S$ and $\xi\in L^2(X,\mu)$, let
  $F_{s,\xi}(x,y):=\xi(y)\chi_{\Gr(s)} (x,y).$ Then $F_{s,\xi}\in
  L^2(R,\nu)$.  A computation (using Lemma~\ref{chop} and similar to
  that in the first paragraph of Theorem~\ref{T: Grepdef}) shows for
  $s_1,\dots, s_N\in\S$ and $\xi_1,\dots, \xi_N\in L^2(X,\mu)$,
\[\norm{\sum_{n=1}^N k_{s_n}\otimes \xi_n}_{\fA\otimes_\pi \H}=
\norm{\sum_{n=1}^N F_{s_n,\xi_n}}_{L^2(R,\nu)}.\]  It follows
there is an isometry $U\in \B(\fA\otimes_\pi \H, L^2(R,\nu))$
which satisfies $U(k_s\otimes \xi)=F_{s,\xi}$. In fact, $U$ is a
unitary operator.

A computation using~\eqref{sigmaformula} shows that for $v\in\G$,
$s\in\S$ and $\xi\in L^2(X,\mu)$, 
\[vF_{s,\xi}=F_{q(v)s,\pi(\sigma(v,s))\xi}.\]  Hence

\begin{align*}
U\lambda_\pi(v) (k_s\otimes \xi)&=U (k_{q(v)s}\sigma(v,x)\otimes \xi)=
U(k_{q(v)s}\otimes\pi(\sigma(v,x))\xi)\\
&=F_{q(v)s,\pi(\sigma(v,s))\xi} =v F_{s,\xi},
\end{align*}
so that 
$U\lambda_\pi(v)U^*=v$, so that $\lambda_\pi$ is unitarily equivalent
to the identity representation, as desired.

\end{remark}

Remark~\ref{unitaryFM} shows that our construction of the
representation $\lambda_\pi$  includes the Feldman-Moore construction as
a special case.
Of course, we have yet to show that the von Neumann algebra generated
by $\lambda_\pi(\P)$ is a Cartan MASA in the von Neumann algebra
generated by $\lambda_\pi(\G)$.  We do this in the next section.

\section{The Cartan pair associated to an extension} 
\label{S: Cpair}
In this section we construct a Cartan pair from an extension.  We will
show in Theorem~\ref{T: indeppsi} that the extension associated to
this Cartan pair is equivalent to the original extension.  Thus,
Theorem~\ref{T: same data} and Theorem~\ref{T: indeppsi} show that
there is a one-to-one correspondence between equivalence classes of
Cartan pairs and equivalence classes of extensions of Cartan inverse
monoids.

Let $\S$ be a Cartan inverse monoid, and let $\P$ be the partial
isometries in $\D:=C(\widehat{\E(\S)})$.  Because $\widehat{\E(\S)}$
is assumed to be a hyperstonean space, $\D$ is $*$-isomorphic to an
abelian von Neumann algebra.    In the sequel, we assume that
$\D$ is an abelian von Neumann algebra.  Let
\begin{equation*} \P\hookrightarrow \G \xrightarrow{q} \S
\end{equation*} be an extension, and fix an order preserving section
$j\colon \S \rightarrow \G$.

We denote by $\fA$ the right Hilbert $\D$-module constructed in
Subsection~\ref{Ss: repo ker}.  Let $\pi$ be a faithful, normal
representation of $\D$, and let
\begin{equation*} \lambda_\pi \colon \G \rightarrow \B( \fA
\otimes_\pi \H)
\end{equation*} be the representation of $\G$ by partial isometries,
as constructed in Theorem~\ref{T: Grepdef} and Corollary~\ref{C:
HilbertSpRep}.

\begin{definition} Let
\begin{equation*} \M_q = (\lambda_\pi(\G))'' \text{ and } \D_q =
(\lambda_\pi(\E(\G))''.
\end{equation*}
\end{definition}

Our goal in this section is to show that $(\M\sge,\D\sge)$ is a Cartan
pair.  The definition of $\M\sge$ and $\D\sge$ depends upon the choice
of $\pi$ and, because $\lambda:\G\rightarrow \L(\fA)$ depends on the
choice of $j$, $\M\sge$ and $\D\sge$ also depend on $j$.  However, we
shall see in Theorem~\ref{T: indeppsi} that the \textit{isomorphism
class} of $(\M\sge,\D\sge)$ depends only on the extension
$\P\hookrightarrow \G\xrightarrow{q}\S$ and not upon $\pi$ or $j$.
We begin by constructing the conditional expectation.

\subsection{A conditional expectation} In this subsection we construct
the faithful, normal conditional expectation from $\M_q$ onto $\D_q$.
This expectation will be constructed from the natural map from $\S$
onto $\E(\S)$: the map
\begin{equation*} s\mapsto s\wedge 1.
\end{equation*} This is an idempotent map from $\S$ onto $\E(\S)$,
which is the identity on $\E(\S)$.

This idempotent map induces an idempotent mapping from $\G$ to $\P$,
which will be the identity on $\P$.  We call this map $\Delta$, and
define it by setting
\begin{equation*} \Delta(v):=v j(q(v)\wedge 1),
\end{equation*} for all $v\in \G$.  First note that
\begin{equation*} q(\Delta(v))=q(v)(q(v)\wedge 1)= q(v)\wedge 1 \in
\E(\S),
\end{equation*} thus $\Delta(v)\in\P$ for all $v\in \G$.  Further, if
$v\in \P$ then $q(v)\in \E(\S)$, thus
\begin{equation*} \Delta(v) = v j(q(v)\wedge 1)=v j(q(v))=v.
\end{equation*}

Our goal now is to show that, given $v\in\G$, the formula,
\[E(\lambda_\pi(v)):=\lambda_\pi(\Delta(v))\] extends to a faithful
conditional expectation $E:\M\sge\rightarrow \D\sge$.  It will take
 a bit more machinery before we can do this.  

Let
\begin{equation*} \fB=\overline{\spn} \{ k_e \colon
e\in\E(\S)\}\subseteq\fA.
\end{equation*} Note that $\fB$ is a right Hilbert $\D$-submodule of
$\fA$.  Proposition~\ref{P: universal1} shows that $\lambda|_{\E(\G)}$
extends to a $*$-monomorphism $\pil:\D\rightarrow \L(\fA)$.  For any
$e,f \in\E(\S)$,
\begin{gather*}
\pil(j(f))k_e=k_{fe}\sigma(j(e),f)=k_{fe}=k_{ef}=k_{e}j(f),
\intertext{and for $s\in\S$,}
\pil(j(f))k_s=k_{s}\, \, j(s^\dag f s).
\end{gather*} It follows that, for any $\xi \in \fB$, $d\in \D$ and
$s\in \S$,
\begin{equation*} \pil(d) \xi=\xi d \dstext{and}\pil(d)k_s=k_s j(s^\dag d s).
\end{equation*} That is, the representation $\pil(\cdot)$, restricted
to $\fB$, is given by the right module action of $\D$ on $\fB$.

\begin{proposition}\label{P:proj in A} For $s\in \S$, the map
$k_s\mapsto k_{s\wedge 1}$ uniquely determines a projection $P \in
\L(\fA)$ with range $\fB$.  Moreover, for each $v\in\G$,
	\begin{equation*} P\lambda(v) P=\lambda(\Delta(v))P.
	\end{equation*}
\end{proposition}

\begin{proof} Let $N\in\bbN$, let $c_1,\dots, c_N\in\bbC$ and
$s_1,\dots , s_N\in\S$.  Put $u=\sum_{n=1}^N c_n k_{s_n}$ and
$v=\sum_{n=1}^N c_n k_{(s_n\wedge 1)}$.  We claim that, as elements of
$\D$,
\begin{equation*}\label{Edef1} 
\innerprod{v,v}\leq \innerprod{u,u}.
\end{equation*} Indeed, let $B\in M_N(\D)$ be the matrix whose $mn$-th
entry is $j((s_m\wedge 1)(s_n\wedge 1))$, let $A\in M_N(\D)$ be the
matrix whose $mn$-th entry is $j(s_n^\dag s_m\wedge 1)=K(s_m,s_n)$, and let $D\in
M_N(\D)$ be the diagonal matrix whose $n$-th diagonal entry is
$s_n\wedge 1$.  Lemma~\ref{L: Kpos} shows that $A\geq 0$,
and it is clear that $D$ is a projection.  Corollary~\ref{C:
Leech} implies that
\[B=AD=DA.\] In particular, $0\leq B\leq A$, so that if $C\in M_{N1}(\D)$ is
the column matrix whose $n1$-th entry is $c_nI$, we obtain
\[\innerprod{v,v}=C^*BC\leq C^*AC=\innerprod{u,u},\] as claimed.

It follows that $k_s\mapsto k_{(s\wedge 1)}$ extends linearly to
contraction $P$ on $\fA$.  Let $s,t\in \S$ and put $e=s\wedge t\wedge
1$.  By Lemma~\ref{C: Leech},
 $e=e^\dag e=t^\dag (s\wedge 1)\wedge 1= s^\dag(t\wedge
1)\wedge 1$.  
Hence, 
\begin{align*} \<P k_s, k_t\> &= \< k_{s\wedge 1}, k_t\> = k_t(s\wedge
  1) = j(t^\dag(s\wedge 1)\wedge 1)\\
  &=j(s^\dag (t\wedge 1)\wedge 1)=k_s(t\wedge 1)= \<
  k_s, k_{t\wedge 1}\> = \<k_s, Pk_t\>.
\end{align*} It follows that $P$ is adjointable.  As $P$ is
idempotent, $P$ is a projection in $\L(\fA)$.  Obviously, $\ran(P)=\fB$.

Let $s\in\S$ and $v\in\G$.  Set $r=q(v)$.  Then,
\begin{align*} P\lambda(v) Pk_s&= Pk_{r(s\wedge 1)} \sigma(v, s\wedge
1) = Pk_{r(s\wedge 1)} j(s\wedge 1) j(r)^\dag v = Pk_{r} j(s\wedge 1)
j(r)^\dag v\\ &=k_{r\wedge 1}j(s\wedge 1) j(r)^\dag v = k_1 j(r\wedge
1)j(s\wedge 1) j(r)^\dag v= k_{s\wedge 1} j(r\wedge 1)v.
\intertext{On the other hand,} \lambda(\Delta(v))Pk_s&=
\lambda(vj(r\wedge 1))k_{s\wedge 1}=\lambda(v)\lambda(j(r\wedge
1))k_{s\wedge 1} =\lambda(v) k_{s\wedge 1} j(r\wedge 1)\\ &=
k_{r(s\wedge 1)}\sigma(v,s\wedge 1) j(r\wedge 1)= k_{r(s\wedge 1)}
j(s\wedge 1) j(r)^\dag v j(r\wedge 1)\\ &= k_{r(s\wedge 1)} j(s\wedge
1) v j(r\wedge 1)= k_{r(s\wedge 1)} v j(r\wedge 1)= k_{(s\wedge 1)}
j(r\wedge 1) v.
\end{align*} Thus $P\lambda(v)Pk_s=\lambda(\Delta(v))Pk_s$.  As this
holds for every $s\in\S$, it follows that
$P\lambda(v)P=\lambda(\Delta(v))P$.
\end{proof}

\begin{lemma}\label{L: V} Define $V:\H\rightarrow \fA\otimes_\pi \H$
by $V\xi := k_1\otimes \xi$.  Then $V$ is an isometry for which the
following properties hold:
\begin{enumerate}
\item for every $s\in \S$ and $\xi\in\H$, $V^*(k_s\otimes \xi)=
\pi(j(s\wedge 1))\xi$;
\item $VV^*=\pi_*(P)$;
\item for every $v\in\G$, $V^*\lambda_\pi(v)V=\pi(\Delta(v)).$
\end{enumerate}
\end{lemma}
\begin{proof} That $V$ is an isometry follows from the fact that
$\innerprod{k_1,k_1}=I\in\D$.  Indeed, for $\xi\in \H$, we have
\[ \innerprod{V\xi, V\xi}=\innerprod{k_1\otimes \xi, k_1\otimes \xi} =
\innerprod{\xi, \pi(\innerprod{k_1,k_1})\xi}=\innerprod{\xi,\xi}.\]
Notice that for $s\in\S$ and $\xi,\eta\in\H$,
\[\innerprod{V\xi, k_s\otimes \eta}= \innerprod{k_1\otimes \xi,
k_s\otimes \eta}=\innerprod{\xi,\pi(k_s(1))\eta}=\innerprod{\xi,
\pi(j(s^\dag \wedge 1))\eta}.\] Since $s^\dag \wedge 1=s\wedge 1$, we
find that $V^*(k_s\otimes \eta)=\pi(j(s\wedge 1))\eta.$ Hence
$VV^*(k_s\otimes \eta)= k_1\otimes \pi(j(s\wedge 1))\eta= k_{s\wedge
1} \otimes \eta= \pi_*(P)(k_s\otimes \eta)$.  So $VV^*=\pi_*(P)$.

By Proposition~\ref{P:proj in A} we have
\begin{equation*} P\lambda(v) P= \lambda(\Delta(v))P.
\end{equation*} Applying $\pi_*$ to each side of this equality yields
\[\pi_*(P)\lambda_\pi(v)\pi_*(P)=
\pi_*(\lambda(\Delta(v)))\pi_*(P)=\pi_*(\pil(\Delta(v)))\pi_*(P).\]  A
calculation gives
$\pi_*(\pil(\Delta(v)))\pi_*(P)=V\pi(\Delta(v))V^*$, so that 
\[\pi_*(P)\lambda_\pi(v)\pi_*(P)=V\pi(\Delta(v))V^*.\]
Part (c) now follows from parts (a) and (b).
\end{proof}

We will now show that $\D$ and $\D_q$ are isomorphic.  Thus, an
expectation onto a faithful image of $\D$ will give rise to an
expectation onto $\D_q$.  We begin with two lemmas.

\begin{lemma}\label{L: copy of D 1} For $d\in \D$ the map
$U\colon\D\rightarrow \fB$ given by $Ud=\pil(d) k_1$ is an isometry of
$\D$ onto $\fB$.
		Furthermore, the map $d\mapsto U^*P\pil(d) U$ is the regular
representation of $\D$ onto itself.
\end{lemma}

\begin{proof} For any $d\in \D$ we have
\begin{equation*} \< Ud, Ud\>=\<\pil(d)k_1,\pil(d)k_1\>=\<k_1 d,k_1
d\>=d^* d.
\end{equation*} Thus, $U$ is an isometry.

To prove the remainder of the Lemma, we note that, for any $d,
h\in\D$, we have
\[U^*P\pil(d) Uh =U^*P\pil(d)
\pil(h)k_1=U^*\pil(dh)k_1=dh.\qedhere \]
\end{proof}

\begin{lemma}\label{L: copy of D 2} For every $d\in\D$,
	\begin{equation*} V\pi(d) V^* = \pi_*(\pil(d))\pi_*(P).
	\end{equation*}
\end{lemma}

\begin{proof} This is a simple calculation.  For $d\in\D$, $s\in \S$
and $\xi\in \H$,
\begin{align*} V\pi(d) V^*(k_s\otimes \xi)&= V\pi(dj(s\wedge 1))\xi=
k_1\otimes \pi(dj(s\wedge1))\xi\\ &=k_1 d \otimes\pi(j(s\wedge 1))\xi=
\pil(d)k_1 \otimes \pi(j(s\wedge 1))\xi\\ &= \pi_*(\pil(d))\left( k_1
\otimes \pi(j(s\wedge 1))\xi\right) = \pi_*(\pil(d))(k_{s\wedge 1}\otimes
\xi)\\ &= \pi_*(\pil(d))\pi_*(P) (k_s\otimes \xi). \qedhere
\end{align*}
\end{proof}

\begin{proposition}\label{P: copy of D} The image of $\D$ under
$\pi_*\circ\pil$ is a von Neumann algebra and the map
$\Phi:\pi(\D)\rightarrow (\pi_*\circ \pil)(\D)$ given by
$\Phi(\pi(d))= \pi_*(\pil(d))$ is an isomorphism of $\pi(\D)$ onto $\D_q$.
\end{proposition}

\begin{proof} Clearly $\Phi$ is a $*$-homomorphism.  Lemma~\ref{L:
    copy of D 1} and Lemma~\ref{L: copy of D 2} show that $\Phi$ is an
  isomorphism of \cstaralg s.  To see that $\pi_*(\pil(\D))$ is a von
  Neumann algebra, it suffices to show that $\pi_*(\pil(\D))$ is
  strongly closed.

For $s\in\S$, the map $\D\ni d\mapsto s^\dag d s \in \D j(s^\dag s)$
is a $*$-homomorphism of the von Neumann algebra $\D$ onto the von
Neumann algebra $\D j(s^\dag s)$ and hence is normal.  Also for
$s\in\S$, $d\in\D$ and $\xi\in\H$,
\begin{equation}\label{Eq: pipil1} \pi_*(\pil(d))(k_s\otimes \xi)=k_s
\otimes \pi(s^\dag ds)\xi,
\end{equation} since
\begin{equation*} \pil(d)k_s=k_s\,\, (s^\dag d s).
\end{equation*}

Let $\N$ denote the strong closure of $\pi_*(\pil(\D))$ and fix
$x\in\N$.  Kaplansky's density theorem ensures that there exists a net
$d_i\in\D$ such that $\norm{d_i}\leq \norm{x}$ and $\pi_*(\pil(d_i))$
converges strongly to $x$.  Equation~\eqref{Eq: pipil1} applied with
$s=1$ implies that $\pi(d_i)$ is a strongly Cauchy net and hence
converges strongly.  Thus, $d_i$ converges $\sigma$-strongly to an element
$d\in\D$.  But another application of equation~\eqref{Eq: pipil1}
shows that $\pi_*(\pil(d_i)) u\rightarrow \pi_*(\pil(d)) u$ for every
$u\in\spn\{k_s\otimes \xi: s\in\S\text{ and } \xi\in\H\}$.  Since
$(d_i)$ is a bounded net, we obtain the strong convergence of
$\pi_*(\pil(d_i))$ to $\pi_*(\pil(d))$.  Hence $x\in \pi_*(\pil(\D))$
as desired.

Finally, for every $e\in\E(\S)$, a calculation gives 
$\pi_*(\pil(j(e)))=\lambda_\pi(j(e))$. Thus
$\pi_*(\pil(\D))=\lambda_\pi(\E(\G))''=\D_q$.
\end{proof}

We are at last ready to define the conditional expectation $E$ from
$\M_q$ onto $\D_q$.
Recall that for $v\in\G$, $V^*\lambda_\pi(v)V=\pi(\Delta(v))\in
\pi(\D)$.  Thus, Proposition~\ref{P: copy of D} shows that the
following definition of $E$ carries $\M_q$ into $\D_q$.

\begin{definition}\label{D: existcondexp} Define the conditional
expectation $E\colon \M_q\rightarrow \D_q$ by
	\begin{equation*} E(x)=\Phi(V^*xV).
	\end{equation*}
\end{definition}

By construction, $E$ is normal, idempotent and
$E|_{\D_q}=$id$_{\D_q}$.  Thus, $E$ is indeed a normal conditional
expectation.  We conclude this subsection by recording some facts
about $E$ that will be useful.

\begin{lemma}\label{L: E calc} For any $v\in\G$ and $x\in\M_q$ we have
\begin{equation*} E(\lambda_\pi(v))=\lambda_\pi(\Delta(v)),
\end{equation*} and
\begin{equation*} E(\lambda_\pi(v)^*x\lambda_\pi(v))=\lambda_\pi(v)^*
E(x)\lambda_\pi(v).
\end{equation*}
\end{lemma}

\begin{proof} The first part follows from the definition of $E$.  For
the second, we will show for $v,w\in\G$ we have
	\[\Delta(w^\dag v w)=w^\dag\Delta(v) w.\] The result will then
follow from the normality of $E$.  Take $v,w\in\G$.  Setting
$r:=q(w)$, we have,
\begin{align*} \Delta(w^\dag v w) &= w^\dag v w \left(j(r^\dag q(v)
r\wedge 1)\right)\\ &= w^\dag v \left(wj(r^\dag q(v) r\wedge 1) w^\dag
\right)w\\ &= w^\dag v \left(j(rr^\dag q(v) rr^\dag\wedge rr^\dag)
\right)w\\ &= w^\dag v \left(j( (q(v) \wedge 1) rr^\dag) \right)w\\ &=
w^\dag v j(q(v)\wedge 1) w= w^\dag \Delta(v)w.\qedhere
\end{align*}
\end{proof}

\subsection{The Cartan pair}\label{Ss: Cpair} Our next goal is to show
that $(\M\sge,\D\sge)$ is a Cartan pair.  That $\D\sge$ is regular in
$\M\sge$ is straightforward.  Much less straightforward is showing
that $\D\sge$ is a MASA in $\M\sge$ and that $E$ is faithful.  The
normality of $E$ and a result of
Kov{\'a}cs-Sz{\aH{u}}cs~\cite[Proposition~1]{KovacsSzucsErTyThvNAl},
imply that if $\D\sge$ is a MASA in $\M_q$, then $E$ is faithful.  On
the other hand, as we shall see below, the fact that $\D\sge$ is a
MASA follows from faithfulness of $E$ and the fact that the set of
normalizing partial isometries span a weak-$*$ dense subset of
$\D\sge$.  Thus there is a \textsl{``Which comes first, $\D\sge$ is a MASA or
$E$ is faithful?''} problem.  It may be possible to give a direct proof
that $\D\sge$ is a MASA, but we will proceed by showing that $E$ is
faithful.

\begin{proposition}\label{P: faithfulexp} The conditional expectation
$E$ is faithful.
\end{proposition}

\begin{proof} Let $\C$ denote the center of $\M\sge$.  We claim that
$E|_\C$ is faithful.  Let $x\in\C$ and suppose $E(x^*x)=0$.  The
definition of $E$ from Definition~\ref{D: existcondexp} and
Proposition~\ref{P: copy of D} show that $xV=0$.  Notice that
$\sigma(j(s),1)=j(s^\dag s)$ so that $\lambda(j(s))k_1 =k_s j(s^\dag
s)=k_s$ (see Corollary~\ref{C: Leech}).  Hence for $s\in\S$ and
$\xi\in \H$,
\begin{equation*} x(k_s\otimes \xi)= x\lambda(j(s))(k_1\otimes \xi)
=\lambda(j(s)) x(k_1\otimes \xi) =\lambda(j(s))xV\xi=0.
\end{equation*} Since the span of such vectors is a dense subspace of
$\fH$, we conclude that $x=0$.

Let $\fJ:=\{x\in\M\sge: E(x^*x)=0\}$.  Then $\fJ$ is a left ideal of
$\M\sge.$ Lemma~\ref{L: E calc} implies that for $x\in\fJ$ and $v\in
\G$, $x\lambda_\pi(v)\in\fJ$.  It now follows that $\fJ$ is a
two-sided ideal of $\M\sge$ as well.  Since $\fJ$ is weak-$*$-closed,
by \cite[Proposition~II.3.12]{TakesakiThOpAlI}, there is a projection
$Q\in\C$ such that $\fJ=Q\M\sge.$ As $Q\in\fJ$ and $E|_\C$ is
faithful, we obtain $Q=0$. Thus $\fJ=(0)$, that is, $E$ is faithful.
\end{proof}

\begin{proposition}\label{P: dmasa} The subalgebra $\D\sge$ is a MASA
in $\M\sge$.
\end{proposition}

\begin{proof} The proof has several preliminary steps.  Let $\D\sge^c$
be the relative commutant of $\D\sge$ in $\M\sge$.

\textit{Step 1:} We first show $\lambda_\pi(\G)\cap \D\sge^c\subseteq
\D\sge$.  To see this, suppose $v\in\G$ and $\lambda_\pi(v)\in
\D\sge^c$.  In particular, $\lambda_\pi(v)$ commutes with every
element of $\lambda_\pi(\E(\G))$.  Since $\lambda_\pi$ is one-to-one,
$v$ commutes with every element of $\E(\G)$.  Since $\S$ is a
fundamental inverse monoid, it follows that $v\in\P$.  Therefore
$\lambda_\pi(v)\in\D\sge$.

\textit{Step 2:} Next, we claim that if $x\in \D\sge^c$, then for
every $v\in\lambda_\pi(\G)$, $vE(v^*x)\in\D\sge$.\footnote{In order
to be consistent with previous notation, we should start with $w\in\G$
and prove $\lambda_\pi(w)E(\lambda_\pi(w)^*x)\in \D\sge$.  But it is
notationally cleaner to write $v:=\lambda_\pi(w)$ instead.  We will
continue to do this when there is little danger of confusion.}  Given
such $x$ and $v$, we have, for each $d\in\D\sge$,
\begin{align*} xd-dx&=0, \quad\text{so}\\
v^*xd-v^*dvv^*x&=0.\quad\text{Apply $E$ to obtain}\\ E(v^*x)d-v^*dv
E(v^*x)&=0;\quad \text{multiplying on the left by $v$ yields}\\
vE(v^*x)d-dvE(v^*x)&=0.
\end{align*} Thus, $vE(v^*x)\in \D\sge^c$.  Let $E(v^*x)=u|E(v^*x)|$
be the polar decomposition of $E(v^*x)$.  Then $u$ is a partial
isometry in $\D\sge$, so $u\in\lambda_\pi(\P)$.  Also, $vu|E(v^*x)|$
is the polar decomposition of $vE(v^*x)$.  As $vE(v^*x)\in\D\sge^c$,
we conclude that $vu\in \lambda_\pi(\G)\cap\D\sge^c$, so by Step 1,
$vu\in\D\sge$.  But $|E(v^*x)|\in\D\sge$, so $vE(v^*x)\in\D\sge$.

\textit{Step 3:} For every $v\in\lambda_\pi(\G)$,
$v-E(v)\in\lambda_\pi(\G)$.  To see this, observe that since
$v^*E(v)\in\D\sge$, we have 
$v^*E(v)=E(v^*E(v))=E(v^*)E(v)$.  As 
$E(v)\in\lambda_\pi(\P)$, we have $I-E(v)^*E(v)\in\lambda_\pi(\P)$.
Hence $\lambda_\pi(\G)\ni v(I-E(v^*)E(v)) = v-E(v),$ as desired.

With these preliminaries completed, we now prove the proposition.  Let
$x\in\D\sge^c$.  If $w\in\lambda_\pi(\G)$ and $E(w)=0$, by Step 2, we
have
\[wE(w^*x)=E(wE(w^*x))=E(w)E(w^*x)=0.\] Multiplying on the left by
$w^*$ shows that $E(w^*x)=0$ whenever $w\in \lambda_\pi(\G)\cap \ker
E$.

By Step 3, we obtain for every $v\in\lambda_\pi(\G)$,
\[E(v^*x)=E((v^*-E(v^*))x)+E(E(v^*)x)=E(v^*)E(x).\] Since $\M\sge$ is
the weak-$*$-closed linear span of $\lambda_\pi(\G)$ and $E$ is
normal, we conclude that for every $x\in \D\sge^c$,
\begin{equation}\label{E: ndone1} E(x^*x)=E(x^*)E(x).
\end{equation} Replacing $x$ by $x-E(x)$ in \eqref{E: ndone1} shows
that for every $x\in\D\sge^c$,
\[E((x-E(x))^*(x-E(x)))=0.\] By faithfulness of $E$, $x=E(x)\in\D\sge$
for every $x\in\D\sge^c$.  This completes the proof.
\end{proof}

We are now ready to show that $(\M\sge,\D\sge)$ is a Cartan pair.

\begin{theorem}\label{T: Cpair} The pair $(\M\sge,\D\sge)$ is a Cartan
pair.
\end{theorem}

\begin{proof} By Proposition~\ref{P: dmasa}, $\D\sge$ is a MASA in
$\M\sge$.  By Proposition~\ref{P: faithfulexp}, there is a faithful
conditional expectation from $\M\sge$ onto $\D\sge$.  Finally, as
\begin{equation*} \lambda_\pi(\G)\subseteq \G\N(\M\sge,\D\sge)
\end{equation*} and the span of $\lambda_\pi(\G)$ is weak-$*$ dense
in $\M\sge$ it follows that $\G\N(\M\sge,\D\sge)$ spans a weak-$*$
dense subset of $\M\sge$.
\end{proof}

We showed in Proposition~\ref{P: S Cartan} and Theorem~\ref{T: same
data} that a Cartan pair uniquely determines an extension by a Cartan
inverse monoid.  To complete our circle of ideas, we now want to show
that the extension for $(\M\sge,\D\sge)$ is equivalent to the
extension
\begin{equation*} \P \hookrightarrow \G \xrightarrow{q} \S.
\end{equation*} from which $(\M\sge,\D\sge)$ was constructed.

\begin{theorem}\label{T: indeppsi} The extension associated to the
Cartan pair $(\M\sge,\D\sge)$ is equivalent to the extension
\begin{equation*}\label{F: indeppsi1}
\P\hookrightarrow\G\xrightarrow{q} \S
\end{equation*} from which $(\M\sge,\D\sge)$ was constructed.

Moreover, the isomorphism class of $(\M\sge,\D\sge)$ depends only upon
the equivalence class of the extension (and not on the choice of
representation $\pi$ or section $j$).
\end{theorem}

\begin{proof} 
Let $R_M$ and $R_{M,\pi}$ be the Munn congruences for $\G$ and
$\lambda_\pi(\G)$ respectively.  Because $\lambda_\pi$ is an
isomorphism of $\G$ onto $\lambda_\pi(\G)$, $(v,w)$ belongs to 
$R_M$ if and only
if $(\lambda_\pi(v),\lambda_\pi(w))$ belongs to $R_{M,\pi}$.  Let
$q_\pi:\lambda_\pi(\G)\rightarrow \lambda_\pi(\G)/R_{M,\pi}$ be the
quotient map.  Then the map $\tilde{\lambda}_\pi:= q_\pi\circ
\lambda_\pi\circ j$ is an isomorphism of $S$ onto
$\lambda_\pi(\G)/R_{M,\pi}$ such that $\tilde{\lambda}_\pi\circ
q=q_\pi\circ \lambda_\pi$.  It is now clear that the extensions
\[\P\hookrightarrow \G\xrightarrow{q}\S\] and
\[\lambda_\pi(\P) \hookrightarrow \lambda_\pi(\G) \xrightarrow{q_\pi}
\widetilde{\lambda}_\pi(\S)\] are equivalent.

Our next task is to show that 
\begin{equation}\label{indeppsi2} \lambda_\pi(\G) = \G\N(\M_q,\D_q).
\end{equation}    It will then follow immediately that  $\lambda_\pi(\P) \hookrightarrow
\lambda_\pi(\G) \xrightarrow{q_\pi} \widetilde{\lambda}_\pi(\S)$ is
the extension associated to $(\M_q,\D_q)$.  

\textit{Claim 1: If $u\in\G\N(\M_q,\D_q)$, then $uE(u^*)$ is a
projection in $\D_q$, and
\begin{equation}\label{E: saE} uE(u^*)=E(uE(u^*))=E(u)E(u^*).
\end{equation}} To see this, suppose $d\in\D_q$.  Then
\[uE(u^*)d=uE(u^*d)=uE(u^*du u^*)=uu^*duE(u^*)=duE(u^*).\] Since
$\D_q$ is a MASA in $\M_q$,  $uE(u^*)\in\D_q.$ Next,
\[uE(u^*)uE(u^*)=uE(u^* uE(u^*))=uu^*u E((E(u^*))=uE(u^*),\] so
$uE(u^*)$ is a projection in $\D_q$.  The equality \eqref{E: saE} is
now obvious.

By construction, $\lambda_\pi(\G)\subseteq \G\N(\M_q,\D_q)$.  To
establish the reverse inclusion, fix $v\in\G\N(\M_q,\D_q)$; without
loss of generality, assume $v\neq 0$.

\textit{Claim 2: There exists $p\in\lambda_\pi(\E(\G))$ such that: a)
  $vp\in \lambda_\pi(\G)$, b) $p\leq v^*v$, and c) $vp\neq 0$.}  Since
$\lambda_\pi(\G)''=\M_q$, it follows (as in the proof of
\cite[Proposition~1.3.4]{CameronPittsZarikianBiCaMASAvNAlNoAlMeTh})
that there exists $w\in \lambda_\pi(\G)$ such that $wE(w^*v)\neq 0$.
Let $p=v^*wE(w^*v).$ By Claim 1, $p\in\D_q$ is a projection.  It is
evident that $p\leq v^*v$.  Moreover, \eqref{E: saE} implies that $
E(v^*w)E(w^*v) = |wE(w^*v))|^2=p,$ so $E(w^*v)$ is a partial isometry in
$\D_q$, so $wE(w^*v)\in\lambda_\pi(\G)$.  Since
$E(w^*v)=w^*v(v^*wE(w^*v)),$ we obtain,
\[ 0\neq wE(w^*v)=w (w^*v(v^*wE(w^*v))) = vv^*wE(w^*v)=vp.\] Thus
Claim 2 holds.

Now let $\F\subseteq \{p\in \D_q: \text{$p$ is a projection and }
p\leq v^*v\}$ be a maximal pairwise orthogonal family of projections
such that for each $p\in\F$, $0\neq vp\in\lambda_\pi(\G)$.  Set
$Q:=\bigvee\F$.  The maximality of $\F$ implies that $Q=v^*v$.
Indeed, if $Q\neq v^*v$, then $Q_1:=v^*v-Q$ is a projection in $\D_q$
and applying Claim 2 to $vQ_1$ yields a projection $v^*v\geq p\in\D_q$
such that $0\neq vp\in\lambda_\pi(\G)$ which is orthogonal to every
element of $\F$.

For each $p\in\F$, set
\[w_p:= \lambda_\pi^{-1}(vp),\quad s_p=q(w_p), \quad h_p=j(s_p)^\dag
w_p \dstext{and} e_p=s_p^\dag s_p.\] Then
\[h_p\in\P,\quad vp = \lambda_\pi(w_p)\dstext{and}
p=\lambda_\pi(j(e_p)).\] Also, $\{s_p: p\in \F\}$ is a pairwise
orthogonal family in $\S$ and hence the sum $\sum_{p\in\F} h_p$
converges weak-$*$ in $\D$.  Let \[s=\bigvee_{p\in \F} s_p,\quad
e:=\bigvee_{p\in \F} e_p, \dstext{and} h=\sum_{p\in \F} h_p.\] Thus,
$h\in\P$ and $h^\dag h=j(s^\dag s)$.  Now set \[w:=j(s) h\in \G.\] We
claim that $\lambda_\pi(w)=v$.  Observe that $v^*v=\lambda_\pi(w^*w)$.
Also, for $p\in \F$, $se_p=s_p$, so
\[\lambda_\pi(w)
p=\lambda_\pi(wj(e_p))=\lambda_\pi(j(s)h_p)=\lambda_\pi(j(s_p)
h_p)=\lambda_\pi(j(s_p)j(s_p)^\dag w_p) =vp.\] Therefore,
\[\lambda_\pi(w)=\lambda_\pi(w) Q = vQ=v.\] Hence
$v\in\lambda_\pi(\G)$.  Therefore
\begin{equation*}
\lambda_\pi(\P) \hookrightarrow
\lambda_\pi(\G) \xrightarrow{q_\pi} \widetilde{\lambda}_\pi(\S)
\end{equation*} is the extension for $(\M_q,\D_q).$

Suppose that $\pi'$ is a faithful normal representation of $\D$ and
that and $j':\S\rightarrow \G$ is an order preserving section for $q$.
Let $(\M'_q,\D'_q)$ be the Cartan pair constructed using $\pi'$ and
$j'$ as in Theorem~\ref{T: Cpair}.  Then the previous paragraphs show
that the extensions associated to $(\M_q,\D_q)$ and $(\M'_q,\D'_q)$
are equivalent extensions.  By Theorem~\ref{T: indeppsi},
$(\M_q,\D_q)$ and $(\M'_q,\D'_q)$ are isomorphic Cartan pairs.  The
proof is now complete.
\end{proof}

\section{The Spectral Theorem for Bimodules and Subdiagonal 
Algebras}
In this section, we provide two illustrations of how our viewpoint may
be used to reformulate and address the validity of a pair of important
assertions found of Muhly, Saito and Solel
\cite{MuhlySaitoSolelCoTrOpAl}.

Muhly, Saito and Solel studied the weak-$*$-closed $\D$-bimodules in a
Cartan pair $(\M,\D)$ as they relate to the underlying equivalence
relation $R$ from the Feldman-Moore construction.  Roughly speaking,
they claimed in \cite[Theorem~2.5]{MuhlySaitoSolelCoTrOpAl} that if
$B\subseteq \M$ is a weak-$*$-closed $\D$-bimodule in $\M$, then there
is a Borel subset $A\subseteq R$ such that $B$ consists of all
operators in $\M$ whose ``matrices" are supported in $B$.  This
statement is commonly known as the Spectral Theorem for Bimodules.  It
has been known for some time that there is a gap in the proof of
\cite[Theorem~2.5]{MuhlySaitoSolelCoTrOpAl}, see
e.g. \cite{AoiCoEqSuInSu}.  When the equivalence relation $R$
 is hyperfinite, the result was
shown to hold by Fulman \cite[Theorem~15.18]{FulmanCrPrvNAlEqReThSu}.
When $\M$ is a hyperfinite factor, $R$ is hyperfinite, .  

An alternate approach to the Spectral Theorem for Bimodules was given
by Cameron, Pitts, and Zarikian in
\cite{CameronPittsZarikianBiCaMASAvNAlNoAlMeTh}.  Rather than
characterizing weak-$*$-closed $\D$-bimodules, Cameron, Pitts and
Zarikian show that the lattice of Bures-closed $\D$-bimodules is
isomorphic to the lattice of projections in a certain abelian von
Neumann algebra $\Z$ associated to the pair $(\M,\D)$, see
\cite[Theorem~2.5.8]{CameronPittsZarikianBiCaMASAvNAlNoAlMeTh}.
Moreover, the work in \cite{CameronPittsZarikianBiCaMASAvNAlNoAlMeTh}
shows that the Spectral Theorem for Bimodules holds if and only if
every weak-$*$-closed $\D$-bimodule in $\M$ is closed in the Bures
topology.  The approach in
\cite{CameronPittsZarikianBiCaMASAvNAlNoAlMeTh} does not rely on the
Feldman-Moore construction.

Our first goal, accomplished in the first subsection, is
to give a description of the Bures-closed $\D$-bimodules in a Cartan
pair $(\M,\D)$ in terms of certain subsets of $\S$, see
Theorem~\ref{T: Spectral Theorem}.  This description of the bimodules
in $\M$ is a direct analogue of the spectral assertion for bimodules
of Muhly, Saito and Solel.  The advantage of the description given in
Theorem~\ref{T: Spectral Theorem} over that in
\cite{CameronPittsZarikianBiCaMASAvNAlNoAlMeTh} is that Bures-closed
bimodules of $\M$ are parametrized in terms of data directly obtained
from the associated extension, so there is no need to consider the
projection lattice of $\Z$.  In Corollary \ref{C: Aoi app}, we use
Aoi's Theorem to refine this result to parametrize the von Neumann
algebras between $\M$ and $\D$.  In the second subsection, we use our
work to give a description of the maximal subdiagonal algebras of $\M$
which contain $\D$, see Theorem~\ref{MSS3.5} below.
Theorem~\ref{MSS3.5} provides a proof of Muhly, Saito, and Solel's
main representation
theorem, 
\cite[Theorem~3.5]{MuhlySaitoSolelCoTrOpAl}, which avoids
the (as yet)  unproven weak-$*$ version of the Spectral Theorem for Bimodules.

\subsection{$\D$-Bimodules and Spectral Sets}

\begin{definition}[\cite{BuresAbSuvNAl}] The \emph{Bures topology} on
$\M$ is the locally convex topology generated by the family of
seminorms
\begin{equation*} \{ T\mapsto\sqrt{\tau(E(T^*T))}:\ \tau\in(\D_*)^+\}.
\end{equation*}
\end{definition}

We define the following subsets of $\S$.
\begin{definition} A subset $A$ of a Cartan inverse monoid $\S$ is a
\emph{spectral set} if
\begin{enumerate}
\item $s\in A$ and $t\leq s$ implies $t\in A$; and
\item $\{s_i\}_{i\in I}$ is a pairwise orthogonal family in $A$, then
$\bigvee_{i\in I}s_i\in A$.
\end{enumerate}
\end{definition} Given two spectral sets $A_1, A_2\subseteq \S$,
define their \textit{join span}, denoted $A_1\js A_2$, to be the set
of all elements of $\S$ which can be written as the join of two
orthogonal elements, one from $A_1$ and the other from $A_2$, that is,
\[ A_1\js A_2:=\{s\in \S: \text{ there exists } s_i\in A_i \text{ such
that } s_1s_2^\dag=s_1^\dag s_2=0 \text{ and } s=s_1\vee s_2\}.
\] It is not hard to see that $A_1\js A_2$ is the smallest spectral
set containing $A_1\cup A_2$.
Thus the spectral sets in $\S$ form a lattice, with join given by
$\js$ and meet given by intersection $\cap$.  We aim to show the
existence of a lattice isomorphism between the spectral sets in $\S$
and the Bures-closed $\D$-bimodules in $\M$.  

For any weak-$*$-closed bimodule $B\subseteq \M$, let
\[\G\N(B,\D):=B\cap \G\N(\M,\D).\] It is shown in
\cite[Proposition~2.5.3]{CameronPittsZarikianBiCaMASAvNAlNoAlMeTh} that
\[\overline{\spn}^{\text{w-}*}(\G\N(B,\D))\subseteq B \subseteq
\overline{\spn}^{\text{Bures}}(\G\N(B,\D)).\] Also, if $B$ is a
Bures-closed $\D$-bimodule, then
$B=\overline{\spn}^{\text{Bures}} 
(\G\N(B,\D))$~\cite[Theorem~2.5.1]{CameronPittsZarikianBiCaMASAvNAlNoAlMeTh}.

For a Bures-closed $\D$-bimodule $B\subseteq \M$, define
$\Theta(B)\subseteq \S$ by
\begin{equation*} \Theta(B)=q(\G\N(B,\D)).
\end{equation*} Further, define a map $\Psi$ from the collection of
spectral sets in $\S$ to Bures-closed $\D$-bimodules in $\M$ by
\[\Psi(A)=\overline{\spn}^{\text{Bures}}(j(A)),\] which is necessarily
a Bures-closed $\D$-bimodule.  

The
following is a restatement of
\cite[Theorem~2.5.8]{CameronPittsZarikianBiCaMASAvNAlNoAlMeTh} in
terms of spectral sets which is in the same spirit as the original
assertion of \cite[Theorem~2.5]{MuhlySaitoSolelCoTrOpAl}.
\begin{theorem}[Spectral Theorem for Bimodules]\label{T: Spectral
Theorem} There is a lattice isomorphism of the lattice of Bures-closed
$\D$-bimodules onto the lattice of spectral sets in $\S$.
\end{theorem}

\begin{proof} Let $B$ be a Bures-closed $\D$-bimodule in $\M$ and let
$A:=\Theta(B)$.  We will first show that $A$ is a spectral set in
$\S$.  Since $B$ is a $\D$-bimodule, if $s\in A$ and $t\leq s$, then
$t\in A$.  Next, suppose that $\{s_i\}_{i\in I}$ is a pairwise
orthogonal family in $A$ and let $s=\bigvee s_i$.  For $i\neq k$, the
orthogonality of $s_i$ and $s_k$ implies that $j(s_i)$ and $j(s_k)$
are partial isometries with orthogonal initial spaces and orthogonal
range spaces.  Therefore, the sum $\sum_{i\in I} j(s_i)$ converges
strong-$*$ to an element $v\in \G\N(\M,\D)$.  As the Bures topology is
weaker than the strong-$*$ topology, $v\in \G\N(B,\D)$.  For every $i\in I$,
$q(vj(s_i^\dag s_i))= s_i$, and it follows that $q(v)=s$.  Thus
$j(s)\in B$, and hence $s\in A$.  Therefore $A=\Theta(B)$ is a
spectral set.

We now prove that $A=\Theta(\Psi(A))$.   Clearly, $A\subseteq
\Theta(\Psi(A))$.  If $A\neq
\Theta(\Psi(A))$, then there exists $t\in \Theta(\Psi(A))$ such that
$t\wedge s=0$ for all $s\in A$.  Thus, suppose $t\in \S$ and $t\wedge
s=0$ for all $s\in A$.  Then $E(j(t)^*j(s))=0$ for all $s\in A$.  It
follows from Corollary 2.3.2 and Lemma 1.4.6 of
\cite{CameronPittsZarikianBiCaMASAvNAlNoAlMeTh}, that $t$ is not in
the Bures-closed bimodule generated by $j(A)$.  Hence
$A=\Theta(\Psi(A)$.

That $\Psi(\Theta(B))=B$ follows from the fact that $B$ is generated
as a $\D$-bimodule by $B\cap \G\N(\M,\D)$. Finally, the order
preserving properties follow by the definitions of $\Theta$ and
$\Psi$.
\end{proof}

Recall that a sub-inverse monoid $\T\subseteq \S$ is \textit{full} if
$\E(\T)=\E(\S)$.  Let
\[ W:=\{\N\subseteq \M: \N \text{ is a von Neumann algebra and }
\D\subseteq \N\subseteq \M\}
\] and let
\[ T:=\{\T\subseteq \S: \T \text{ is a full Cartan inverse submonoid}
\}.\] It follows by Aoi's Theorem~\cite{AoiCoEqSuInSu} that if $\M$
has a separable predual, then for any $\N\in W$, $(\N,\D)$ forms a
Cartan pair.  Cameron, Pitts and Zarikian give an alternative proof of
Aoi's Theorem
\cite[Theorem~2.5.9]{CameronPittsZarikianBiCaMASAvNAlNoAlMeTh}.  Their
approach shows that every von Neumann algebra $\N$ with
$\D\subseteq\N\subseteq \M$ is Bures-closed and does not require that
$\M$ has a separable predual.  We note that, while Aoi's original
approach relied on the Feldman-Moore construction of Cartan pairs, the
proof in \cite{CameronPittsZarikianBiCaMASAvNAlNoAlMeTh} is
independent of the work of Feldman and Moore.  The following corollary
to Theorem~\ref{T: Spectral Theorem} is immediate.

\begin{corollary}\label{C: Aoi app} The map $\Theta|_W$ is a bijection of
$W$ onto $S$ and $\Theta_W^{-1}=\Psi|_S$. 
\end{corollary}

\subsection{Subdiagonal Algebras} If $\N$ is a von Neumann algebra
such that $\D\subseteq \N\subseteq \M$,
\cite[Theorem~2.5.9]{CameronPittsZarikianBiCaMASAvNAlNoAlMeTh} shows
that $\N$ is Bures-closed and there exists a unique Bures-continuous,
faithful conditional expectation $\Phi_\N:\M\rightarrow\N$.
 
We record the following two lemmas.  We first show that under certain
circumstances, the Bures closure of an algebra is again an algebra.
Then we show that given a von Neumann algebra $\D\subseteq \N
\subseteq \M$, the conditional expectation $\Phi_\N$ is multiplicative
on certain subalgebras of $\M$.

\begin{lemma}\label{L: Balg1} Suppose $\A$ is a weak-$*$-closed
subalgebra of $\M$ containing $\D$, and let $\N:=\A\cap \A^*$.  Then
the Bures closure of $\A$ is a subalgebra of $\M$ and
$\N=\overline{\A}^{\text{Bures}}\cap
(\overline{\A}^{\text{Bures}})^*$.
\end{lemma}

\begin{proof} Let $\B$ be the Bures closure of $\A$ and choose
$X\in\B$.
By~\cite[Theorem~2.5.1]{CameronPittsZarikianBiCaMASAvNAlNoAlMeTh},
there exists a net $X_\lambda\in\spn\G\N(\A,\D)$ such that
$\text{Bures-}\lim X_\lambda =X$.  Let $v\in\G\N(\A,\D)$.  Since
$E(v^*(X_\lambda-X)^*(X_\lambda-X)v)=v^*E((X_\lambda-X)^*(X_\lambda-X))v$,
it follows that $\text{Bures-}\lim (X_\lambda v)=Xv$. Thus $Xv\in\B$.

Now suppose that $Y\in \B$.  We may write $Y=\text{Bures-}\lim
Y_\lambda$, where $Y_\lambda\in \spn(\G\N(\A,\D))$.  Then $XY_\lambda
\in \B$ for every $\lambda$.  Moreover, the estimate,
\[E((X(Y-Y_\lambda))^*(X(Y-Y_\lambda)))\leq \norm{X}^2
E((Y-Y_\lambda)^*(Y-Y_\lambda))\] implies that $XY_\lambda$ Bures
converges to $XY$, so $XY\in\B$.  Thus $\B$ is an algebra.

By \cite[Theorem~2.5.1]{CameronPittsZarikianBiCaMASAvNAlNoAlMeTh},
$\A\cap \G\N(\M,\D)=\B\cap \G\N(\M,\D)$.  Therefore, $\A\cap
\A^*\cap\G\N(\M,\D)= \B\cap \B^*\cap \G\N(\M,\D)$.  But $\N$ is the
Bures closure of $\A\cap\A^*\cap \G\N(\M,\D)$, so $\A\cap \A^*=\B\cap
\B^*$.
\end{proof}

\begin{lemma}\label{L: Balg2} Suppose $\A$ is a Bures-closed
subalgebra of $\M$ containing $\D$, and let $\N:=\A\cap \A^*$.  Then
for $X, Y\in\A$, $\Phi_\N(XY)=\Phi_\N(X)\Phi_\N(Y)$
\end{lemma}

\begin{proof} Let $\J:=\ker(\Phi_\N |_\A)$.  We shall show that
$\J\cap \G\N(\A,\D)$ is a semigroup.

Suppose first that $u, v\in\G\N(\A,\D)$, that $0\in\{\Phi_\N (u),
\Phi_\N (v)\}$, and $uv\in\N$.  We claim that $uv=0$.  To see this,
suppose that $\Phi_\N (u)=0$.  As $\N$ is closed under adjoints,
$\A\ni v(v^*u^*)=(vv^*)u^*\in\A^*$, so $vv^*u^*\in\N$.  Hence
\[vv^*u^*=\Phi_\N (vv^*u^*)=vv^*\Phi_\N (u)^*=0.\] It follows that
$v^*u^*=uv=0$.  A similar argument shows that $uv=0$ under the
assumption that $\Phi_\N (v)=0$, so the claim holds.
  
Now let $u, v\in\J\cap \G\N(\A,\D)$.  By
\cite[Lemma~2.3.1(a)]{CameronPittsZarikianBiCaMASAvNAlNoAlMeTh}, there
exists $p\in\text{proj}(\D)$ such that $uvp=\Phi_\N (uv)$.  The claim
applied to $u$ and $vp$ shows that $uvp=0$, so $\J\cap \G\N(\A,\D)$ is
a semigroup.

Let $\A_0=\spn\G\N(\A,\D)$.  For $i=1,2$, let $X_i\in\A_0$.  Then
$\Phi_\N(X_i)\in \A_0$.  Write $X_i=\Phi_\N(X_i)+ Y_i$ where
$Y_i=X_i-\Phi_\N(X_i)\in \spn(\J\cap \G\N(\A,\D))$.  Since $\J\cap
\G\N(\A,\D)$ is a semigroup, $\spn(\J\cap \G\N(\A,\D))$ is an algebra.
Then
\begin{align*}\label{6.5.b} \Phi_\N(X_1X_2)&=\Phi_\N(X_1)\Phi_\N(X_2)+
\Phi_\N(\Phi_\N(X_1)Y_2+ Y_2\Phi_\N(X_2)) +\Phi_N(Y_1Y_2)\\
&=\Phi_\N(X_1)\Phi_\N(X_2).
\end{align*}

As $\Phi_\N$ is Bures continuous, the previous equality also holds for
$X_i\in\overline{\A_0}^\text{Bures}$, and we are done.
\end{proof}

\begin{definition} Let $\A$ be a weak-$*$-closed subalgebra of $\M$
such that $\D\subseteq \A\subseteq \M$, and put $\N=\A\cap \A^*$.
Then
\begin{enumerate}
\item $\A$ is \textit{subdiagonal} if $\A+\A^*$ is weak-$*$ dense in
$\M$ and $\Phi_\N|_\A$ is multiplicative;
\item $\A$ is \textit{maximal subdiagonal} if there is no subdiagonal
subalgebra $\B$ of $\M$ with $\B\cap \B^*=\A\cap \A^*$ which properly
contains $\A$;
\item $\A$ is \textit{triangular} if $\A$ is subdiagonal and $\A\cap
\A^*=\D$; and
\item $\A$ is \textit{maximal triangular} if there is no triangular
subalgebra $\B$ of $\M$ with $\B\cap \B^*=\D$ which properly contains
$\A$.
\end{enumerate}
\end{definition}

The following is an immediate consequence of Lemma~\ref{L: Balg1} and
Lemma~\ref{L: Balg2}.

\begin{corollary} \label{BalgCor} If $\A$ is a subdiagonal subalgebra
of $\M$ containing $\D$, then $\overline{\A}^{\text{Bures}}$ is a
subdiagonal algebra with $\A\cap \A^*=
\overline{\A}^{\text{Bures}}\cap
\left(\overline{\A}^{\text{Bures}}\right)^*$.  In particular, every
maximal subdiagonal algebra $\A$ with $\D\subseteq \A\subseteq \M$ is
Bures-closed.
\end{corollary}

Muhly, Saito, and Solel assert that any subdiagonal algebra containing
$\D$ is maximal
subdiagonal.  As their proof (\cite[p.\
263]{MuhlySaitoSolelCoTrOpAl})
depends on the spectral theorem for
weak-$*$-closed bimodules, their assertion remains open.  However,
because maximal subdiagonal algebras are Bures-closed, it
is possible to modify their ideas to
 give descriptions of the maximal subdiagonal and maximal
triangular subalgebras of $\M$ which contain $\D$.  To do this, some
notation is helpful.   A submonoid of $\S$ which is also a
spectral set is a
\textit{spectral monoid}.  Let
\begin{align*}\msd(\S)&:=
\{A\subseteq \S: A \text{ is a spectral monoid 
    containing $\E(\S)$ and } 
A\js A^\dag=\S\}  \\
\intertext{and} \mtr(\S)&:=\{A\in \msd(\S): A\cap A^\dag=\E(\S)\}.
\end{align*}

\begin{remark} The sets $\mtr(\S)$ and $\msd(\S)$ correspond to the
sets $\fP$ and $\fP'$ of \cite[p.~258 and
262]{MuhlySaitoSolelCoTrOpAl}, respectively.
\end{remark}

\begin{theorem}\label{MSS3.5} The restriction of $\Psi$ to $\msd(\S)$
gives a bijection of $\msd(\S)$ onto the the set of all maximal
subdiagonal algebras in $\M$ containing $\D$.  In addition, the
restriction of $\Psi$ to $\mtr(\S)$ is a bijection of $\mtr(\S)$ onto
the set of all weak-$*$-closed maximal triangular subalgebras of $\M$
containing $\D$.
\end{theorem}
\begin{proof} Let $A\in \msd(\S)$.  Since $A\js A^\dag=\S$,
$\G\N(\M,\D)\subseteq \Psi(A)+\Psi(A)^*$, so $\Psi(A)+\Psi(A)^*$ is
weak-$*$ dense in $\M$.  Thus, Lemma~\ref{L: Balg1} and Lemma~\ref{L:
Balg2} shows that $\Psi(\A)$ is a Bures-closed subdiagonal algebra.

Suppose $\B\subseteq \M$ is a subdiagonal algebra with $\B\cap
\B^*=\Psi(A)\cap\Psi(\A)^*$ and $\Psi(A)\subseteq \B$.  If
$u\in\G\N(\B,\D)$, then we may find orthogonal elements $s_1\in A$ and
$s_2\in A^\dag$ such that $q(u)=s_1\vee s_2$.  Then $u=w_1+w_2$, where
$w_i=uj(s_i^\dag s_i)$.  As $\D\subseteq \B$, $w_2\in\B$.  On the
other hand, $q(w_2)^\dag=s_2^\dag\in A$, so $w_2^*\in \Psi(A)^*\subseteq \B^*$,
hence $w_2\in \B\cap \B^*\subseteq \Psi(A)$.  As $w_1\in\Psi(A)$, we
obtain $u\in\Psi(A)$.  Therefore, $\G\N(\B,\D)\subseteq \Psi(A)$.  We
then obtain $\B\subseteq
\overline{\spn}^{\text{Bures}}(\G\N(\B,\D))\subseteq \Psi(A)$.  Thus
$\Psi(A)=\B$, so $\Psi(A)$ is maximal subdiagonal.

On the other hand, suppose $\A\subseteq \M$ is a maximal subdiagonal
algebra containing $\D$.  Set $\N:=\A\cap \A^*$ and let
$A:=\Theta(\A)$.  Since $\D\subseteq \A$, $\E(\S)\subseteq A$;
moreover, $A$ is a monoid because $q$ is a homomorphism and
$\G\N(\A,\D)$ is a monoid.  We need to show that $\S= A\js A^\dag$.

To do this, let $s\in\S$ and set $v=j(s)$.  Using
\cite[Lemma~2.3.1(a)]{CameronPittsZarikianBiCaMASAvNAlNoAlMeTh} twice,
there exist projections $p_+, p_-\in\text{proj}(\D)$ such that:
\begin{enumerate}
\item[i)] $vp_+\in \A$ and $vp_+^\perp$ is $\D$-orthogonal to $\A$;
and
\item[ii)] $vp_+^\perp p_-\in \A^*$ and $vp_+^\perp p_-^\perp$ is
$\D$-orthogonal to $\A^*$.
\end{enumerate} Then $vp_+^\perp p_-^\perp$ is $\D$-orthogonal to
$\A+\A^*$ and hence $\D$-orthogonal to $\M$.  Therefore, $vp_+^\perp
p_-^\perp=0$, so that $vp_+^\perp =vp_+^\perp p_-\in \A^*$.  Then
$s=q(v)=q(vp_+)\vee q(vp_+^\perp)\in A\js A^\dag$.
 Thus, $\Theta(\A)\in \msd(\S)$.

By Theorem~\ref{T: Spectral Theorem}, the restriction of $\Psi$ to the
class of maximal subdiagonal algebras containing $\D$ is a bijection
onto $\msd(\S)$.

It is easy to see that for any maximal triangular algebra $\A$
containing $\D$, $\Psi(\A)\in \mtr(\S)$ and that if $A\in \mtr(\S)$,
then $\Theta(A)$ is a maximal triangular algebra.  Thus the
restriction of $\Psi$ to the class of maximal triangular algebras
containing $\D$ is a bijection onto $\mtr(\S)$.
\end{proof}

\bibliographystyle{amsplain}

\def\cprime{$'$}
\providecommand{\bysame}{\leavevmode\hbox to3em{\hrulefill}\thinspace}
\providecommand{\MR}{\relax\ifhmode\unskip\space\fi MR }

\end{document}